\documentclass[a4paper,11pt,british,oneside]{article}

\usepackage[utf8]{inputenc}
\usepackage[dvipsnames,svgnames,x11names,hyperref]{xcolor}
\usepackage{amsmath,amssymb,amsthm,enumerate,textgreek,hyphenat}
\usepackage[nosort,nocompress,noadjust]{cite}

\usepackage[bookmarks,pdfpagelabels,hyperfootnotes=false,colorlinks,
    linkcolor={red!60!black},
    citecolor={blue!50!black},
    urlcolor={blue!80!black}]{hyperref}

\numberwithin{thcounter}{section}
\numberwithin{equation}{section}

\renewcommand{\eqref}[1]{\hyperref[#1]{(\ref{#1})}}

{\theoremstyle{definition}\newtheorem{definition}{Definition}[section]

\newtheorem{remark}[definition]{Remark}
\newtheorem{question}[definition]{Question}
}

\newtheorem{proposition}[definition]{Proposition}
\newtheorem{lemma}[definition]{Lemma}
\newtheorem{theorem}[definition]{Theorem}

\newtheorem{letterthm}{Theorem}

\renewcommand{\bar}{\overline}
\renewcommand{\tilde}{\widetilde}
\renewcommand{\hat}{\widehat}

\providecommand\numberthis{\addtocounter{equation}{1}\tag{\theequation}}

\providecommand{\NN}{\mathbb{N}}
\providecommand{\ZZ}{\mathbb{Z}}
\providecommand{\QQ}{\mathbb{Q}}
\providecommand{\CC}{\mathbb{C}}
\providecommand{\FF}{\mathbb{F}}
\providecommand{\bF}{\mathbf{F}}
\providecommand{\counit}{\varrho}
\providecommand{\Amp}{\mathbf{A}}
\providecommand{\SU}{\mathrm{SU}}
\providecommand{\NC}{\mathrm{NC}}
\providecommand{\PGL}{\mathrm{PGL}}

\providecommand{\KK}{\mathbb{K}}
\providecommand{\GG}{\mathbb{G}}
\providecommand{\Tr}{\operatorname{Tr}}
\providecommand{\Aut}{\operatorname{Aut}}

\providecommand{\onb}{\operatorname{onb}}

\providecommand{\Irr}{\operatorname{Irr}}
\providecommand{\mult}{\operatorname{mult}}
\providecommand{\id}{\mathrm{id}}
\providecommand{\Rep}{\operatorname{Rep}}
\providecommand{\AutTr}{\Aut^{\mathrm{tr}}}
\providecommand{\recht}{\rightarrow}
\providecommand{\cC}{\mathcal{C}}
\providecommand{\Stab}{\operatorname{Stab}}
\providecommand{\eps}{\varepsilon}
\providecommand{\cC}{\mathcal{C}}
\providecommand{\ot}{\otimes}
\providecommand{\al}{\alpha}
\providecommand{\be}{\beta}
\providecommand{\cU}{\mathcal{U}}
\providecommand{\actson}{\curvearrowright}
\providecommand{\cH}{\mathcal{H}}
\providecommand{\cK}{\mathcal{K}}
\providecommand{\cS}{\mathcal{S}}

\providecommand{\cL}{\mathcal{L}}
\providecommand{\PG}{\operatorname{PG}}
\providecommand{\cO}{\mathcal{O}}
\providecommand{\Hilb}{\operatorname{Hilb}}
\providecommand{\HH}{\mathbb{H}}
\providecommand{\ovt}{\mathbin{\overline{\otimes}}}
\providecommand{\om}{\omega}
\providecommand{\supp}{\operatorname{supp}}
\providecommand{\Ad}{\operatorname{Ad}}
\providecommand{\Om}{\Omega}
\providecommand{\SL}{\operatorname{SL}}
\providecommand{\cA}{\mathcal{A}}
\providecommand{\cG}{\mathcal{G}}
\providecommand{\Out}{\operatorname{Out}}

\providecommand{\cV}{\mathcal{V}}

\usepackage{tikz,textgreek}
\usetikzlibrary{calc,decorations,decorations.pathreplacing}
\newcommand{\roundedbox}[4]{
    \draw[rounded corners=5pt, thick, fill=white] ($#1+(-#2,-#2)+(-#3,0)$) rectangle ($#1+(#2,#2)+(#3,0)$);
    \coordinate (ZZa) at ($#1+(-#3,0)$);
    \coordinate (ZZb) at ($#1+(#3,0)$);
    \node at ($1/2*(ZZa)+1/2*(ZZb)$) {#4};
}

\usepackage{tikz}

\begin{document}

\begin{center}
    {\LARGE\boldmath\bf Property~(T) discrete quantum groups and\vspace{0.5ex}\\ subfactors with triangle presentations}

    \bigskip

    {\sc by Stefaan Vaes\footnote{\noindent KU~Leuven, Department of Mathematics, Leuven (Belgium). E-mails: stefaan.vaes@kuleuven.be and matthias.valvekens@kuleuven.be.}\textsuperscript{,}\footnote{\noindent Supported by European Research Council Consolidator Grant 614195, and by long term structural funding~-- Methusalem grant of the Flemish Government.} and Matthias Valvekens\textsuperscript{1,2}}
\end{center}

\begin{abstract}
\noindent Kazhdan's property~(T) has been studied for several discrete group-like structures, including standard invariants of Jones' subfactors and discrete quantum groups. We prove a \.Zuk-type spectral gap criterion for property~(T) in this setting. This allows us to construct a family of property~(T) discrete quantum groups, as well as subfactor standard invariants, given by generators and a triangle presentation. These are the first examples of property~(T) discrete quantum groups that are not directly related to a discrete group with property~(T).
\end{abstract}

\section{Introduction}

Quantum versions of discrete groups appear in several different contexts, in particular as standard invariants of Jones' subfactors and as duals of Woronowicz' compact quantum groups, including as $q$-deformed enveloping algebras. In each of these settings, Kazhdan's property~(T) has been introduced and plays an important role.

The basic construction associates to any finite index subfactor $N \subset M$ the Jones tower of factors $N \subset M \subset M_1 \subset M_2 \subset \cdots$. The standard invariant of $N \subset M$ is given by the relative commutants $M_i' \cap M_j$, which are finite direct sums of matrix algebras, together with the Jones projections belonging to these relative commutants. This intricate combinatorial invariant was axiomatized in \cite{popa-lambda-lattices} as a $\lambda$-lattice and in \cite{jones:planar} as a planar algebra.

The point of view to interpret the standard invariant $\cG$ of a subfactor $N \subset M$ as a discrete group-like structure acting on $M$ has been very useful. In particular, several notions from representation theory and harmonic analysis have been introduced for standard invariants and they play a central role in the theory. Most notably in \cite{popa-classification-amenable,popa-symmenv}, amenability of standard invariants was defined and used to prove that the standard invariant is a complete invariant for amenable hyperfinite subfactors.

Kazhdan's property~(T) was defined for standard invariants in \cite{popa-symmenv}. Examples of property~(T) subfactors were given in \cite{bisch-popa}, of the form $M^H \subset M \rtimes K$, where $H$ and $K$ are finite groups acting by outer automorphisms on the II$_1$ factor $M$ and generating a property~(T) subgroup of the outer automorphism group $\Out(M)$. A variant of this construction, based on a locally compact property~(T) group $G$ generated by compact open subgroups $H,K < G$ was introduced in \cite{arano-vaes}.

The first property~(T) standard invariants not directly related to property~(T) groups were only constructed recently in \cite{pv-repr-subfactors}, making use of \cite{arano-suqn} and the connection with the representation theory of the quantum groups $\SU_q(n)$, $n \geq 3$, as we discuss below.

Parallel to the development of subfactors and motivated by Tannaka--Krein duality for compact groups and by the $q$-deformations of compact Lie groups in \cite{drinfeld,jimbo}, compact quantum groups were introduced in \cite{woron-su2,woron-tkdual,woron-cqg}. The dual of a compact quantum group is a discrete quantum group: a direct sum of matrix algebras equipped with a comultiplication satisfying a natural set of axioms. When each of these matrix algebras is one-dimensional, we recover the algebra $\ell^\infty(\Gamma)$ of bounded functions on a discrete group $\Gamma$ and the comultiplication is given by dualizing the multiplication on $\Gamma$. An intrinsic definition for discrete quantum groups was given in \cite{van-daele}.

Although property~(T) for discrete quantum groups was defined in \cite{fima-prop-T} much in the same way as for discrete groups, until now, there were no genuinely quantum examples. The known examples were either twists of property~(T) groups in \cite{fima-prop-T} or equal up to finite index to a property~(T) group in \cite{fima-bicrossed}. The main goal of this paper is to construct such genuinely quantum examples of property~(T) discrete quantum groups.

A unifying point of view on both subfactor standard invariants and discrete quantum groups is given by the concept of a rigid $C^*$-tensor category (see Section \ref{sec.tensor-category} for basic terminology and \cite{neshveyev-tuset} for a comprehensive introduction).
%
%
The easiest examples of rigid $C^*$-tensor categories are the category $\Rep(K)$ of finite-dimensional unitary representations of a compact group $K$ and the trivial example $\Hilb_f$, the category of finite-dimensional Hilbert spaces. Similarly, the category $\Rep(\GG)$ of finite-dimensional unitary representations of a compact quantum group $\GG$ is a rigid $C^*$-tensor category.

When $N \subset M$ is a finite index subfactor, the category $\cC$ of all $M$-bimodules that arise as an $M$-subbimodule of some $M_n$ in the Jones tower $N \subset M \subset M_1 \subset M_2 \subset \cdots$, together with the relative tensor product of $M$-bimodules, is a rigid $C^*$-tensor category. By \cite{popa-lambda-lattices}, every finitely generated rigid $C^*$-tensor category arises in this way.

When $\cC = \Rep(\GG)$ is the representation category of a compact quantum group, one associates to every unitary representation its carrier Hilbert space. This provides the fiber functor $\Rep(\GG) \recht \Hilb_f$. The Tannaka--Krein theorem of \cite{woron-tkdual} says that also the converse holds: given a rigid $C^*$-tensor category $\cC$ and a fiber functor $\cC \recht \Hilb_f$, there is a canonical compact quantum group $\GG$ such that $\cC = \Rep(\GG)$.

An intrinsic definition of property~(T) for rigid $C^*$-tensor categories was given in \cite{pv-repr-subfactors}, in the context of a general unitary representation theory for such categories. When $\cC$ is the category generated by a subfactor $N \subset M$, it is proved in \cite{pv-repr-subfactors} that property~(T) of $\cC$ is equivalent with property~(T) for the standard invariant of $N \subset M$ as defined in \cite{popa-symmenv}. When $\cC = \Rep(\GG)$ is the representation category of a compact quantum group $\GG$, it is also proved in \cite{pv-repr-subfactors} that property~(T) of $\cC$ is equivalent with the central property~(T) for the discrete quantum group $\hat{\GG}$ as defined in \cite{arano-suqn}. The results in \cite{arano-suqn,corey-jones,arano-drinfeld-unitary-duals} then say that the representation categories of $\SU_q(n)$, $n \geq 3$, and more generally, of all $q$-deformations of higher rank compact Lie groups have property~(T).

Altogether this provides numerous very interesting classes of rigid $C^*$-tensor categories $\cC$ with property~(T). But it is still wide open whether any of these representation categories can be used to construct discrete quantum groups with property~(T). The problem is the following. Only when $\GG$ is of Kac type, i.e.\ when the Haar state on $\GG$ is a trace, the central property~(T) for $\hat{\GG}$ is equivalent with the actual property~(T) for $\hat{\GG}$. So in order to construct property~(T) discrete quantum groups, one should find \emph{dimension preserving} fiber functors $\Rep(\SU_q(n)) \recht \Hilb_f$ for some $n \geq 3$ and $q \neq -1,0,1$. When $n = 2$, such fiber functors exist for the appropriate values of $q$ and give rise to the universal orthogonal quantum groups $A_o(k)$. When $n = 3$, the existence of such a fiber functor is a very interesting open problem that is equivalent to finding new and exotic Hecke symmetries (see \cite{gurevich}). In this paper, we introduce a new class of rigid $C^*$-tensor categories that have property~(T) and that admit explicit dimension preserving fiber functors to $\Hilb_f$.

The two main results of this paper are the following. In the first part, we prove a quantum version of \.Zuk's spectral gap criterion for property~(T) in \cite{zuk-propT}. Given a countable group $\Gamma$ and a symmetric finite generating set $S \subset \Gamma \setminus \{e\}$, the \emph{link} is the finite graph with vertex set $S$ and with an edge between $g,h \in S$ if and only if $g^{-1} h \in S$. By \cite{zuk-propT}, if the link is connected and if the first positive eigenvalue of its combinatorial Laplacian is larger than $1/2$, then $\Gamma$ has property~(T).

\begin{letterthm}[{see Theorem \ref{thm:spectral-gap-criterion}}]\label{thm-A}
Let $\cC$ be a rigid $C^*$-tensor category and $S \subset \Irr(\cC) \setminus \{\eps\}$ a symmetric finite generating set. We canonically define a finite-dimensional Hilbert space $\cV$ in \eqref{eqn:vertex-space} and a Laplace operator $\Delta$ on $\cV$ in \eqref{eqn:laplacian-def} so that the following holds: if $0$ is a simple eigenvalue of $\Delta$ and if the next eigenvalue is larger than $1/2$, then $\cC$ has property~(T).
\end{letterthm}

The main tool to prove Theorem \ref{thm-A} is the tube algebra $\cA$ of a rigid $C^*$-tensor category $\cC$, as introduced in \cite{ocneanu-chirality} (see Section \ref{sec:tube-algebra}). By \cite{ghosh-jones,neshveyev-yamashita}, property~(T) of $\cC$ is equivalent with the trivial representation of $\cA$ being isolated among all Hilbert space representations of $\cA$. This is a crucial ingredient in the proof of Theorem \ref{thm-A}.

The advantage of Theorem \ref{thm-A} is that it provides a verifiable and computable sufficient condition for property~(T). In Proposition \ref{prop:apply-to-su-q-3}, we show that the criterion is satisfied when $\cC = \Rep(\SU_q(3))$ and $S = \{u,\bar{u}\}$, where $u$ is the fundamental representation of $\SU_q(3)$. This provides an elementary proof for one of the results in \cite{arano-suqn}.

We then turn to the construction of a new class of property~(T) standard invariants and property~(T) discrete quantum groups, using a generators and relations approach. The simplest examples where \.Zuk's original criterion is satisfied are given by the triangle presentations of \cite{cmsz1}. Given a finite set $F$, a triangle presentation is a subset $T \subset F \times F \times F$ satisfying a natural set of conditions that we recall in Definition \ref{def:triangle-pres} below. One then defines the group $\Gamma_T$ with generators $(a_x)_{x \in F}$ and relations $a_x a_y a_z = e$ for all $(x,y,z) \in T$. In \cite{zuk-propT}, it is proved that $\Gamma_T$ together with the symmetric generating set $S = \{a_x,a_x^{-1} \mid x \in F\}$ satisfies the spectral gap criterion.

In Definition \ref{def:G-T}, we define for every triangle presentation $T \subset F \times F \times F$, the compact quantum group $\GG_T$ whose underlying $C^*$-algebra $C(\GG)$ is the universal $C^*$-algebra generated by elements $(u_{xy})_{x,y \in F}$ satisfying the following two relations: $u$ is a unitary element of $B(\CC^F) \ot C(\GG)$ and
$$\sum_{(x,y,z) \in T} \; u_{ax} \; u_{by} \; u_{cz} = \begin{cases} 1 &\quad\text{if $(a,b,c) \in T$,}\\ 0 &\quad\text{if $(a,b,c) \not\in T$.}\end{cases}$$
By construction, $\GG$ is of Kac type, so that the representation category $\Rep(\GG)$ canonically comes with a dimension preserving fiber functor to $\Hilb_f$. Also by construction, the discrete group $\Gamma_T$ is a quotient of the discrete quantum group $\widehat{\GG_T}$.

By \cite{cmsz1}, the Cayley graph of $\Gamma_T$ is the $1$-skeleton of an $\tilde{A}_2$-building $\Delta_T$. In many cases, although not in all cases, this $\tilde{A}_2$-building is \emph{classical}, i.e.\ the Bruhat--Tits $\tilde{A}_2$-building of a local field (see Section \ref{sec:bruhat-tits}). Our second main result is then the following.

\begin{letterthm}[{see Theorem \ref{thm:repgt-main-theorem}}]\label{thm-B}
When $T \subset F \times F \times F$ is a triangle presentation such that $\Delta_T$ is the Bruhat--Tits $\tilde{A}_2$-building of a commutative local field, the discrete quantum group $\widehat{\GG_T}$ and the rigid $C^*$-tensor category $\Rep(\GG_T)$ both have property~(T).
\end{letterthm}

In Remark \ref{rem:planar-algebra}, we explain how to associate a planar algebra or $\lambda$-lattice to a triangle presentation $T$. Under the same conditions as in Theorem \ref{thm-B}, these $\lambda$-lattices have property~(T) as well.

In Section \ref{sec:classification}, we classify the compact quantum groups $\GG_T$ up to isomorphism, in terms of isomorphisms between triangle presentations. Note here that in \cite{cmsz1,cmsz2}, a complete classification of triangle presentations of low order was given. In combination with Theorem \ref{thm:classif-compact}, we thus find that our construction gives rise to numerous non-isomorphic discrete quantum groups $\widehat{\GG_T}$.

An important tool in proving Theorem \ref{thm-B} is the locally compact group $G = \AutTr(\Delta_T)$ of type-rotating automorphisms of the building $\Delta_T$, together with the compact open subgroup $K$ given by the stabilizer of the origin $e \in \Delta_T$. We construct a monoidal functor from $\Rep(\GG_T)$ to the rigid C$^*$-tensor category $\mathcal{C}_f(K<G)$ associated in \cite{arano-vaes} to any compact open subgroup $K$ of a locally compact group $G$. We use this monoidal functor to analyze the low-order tensor powers of the fundamental representation of $\GG_T$.

Since $\Delta_T$ is also the Cayley graph of the discrete group $\Gamma_T$ defined as above by the triangle presentation $T$, we have $\Gamma_T < G$ as a discrete subgroup. By construction, $G = \Gamma_T \, K$ forms a matched pair of the discrete group $\Gamma_T$ and the compact group $K$, in the sense of \cite{Kac,baaj-skandalis,vaes-vainerman,fima-bicrossed}. The bicrossed product construction associates to this matched pair the compact quantum group $\HH_T$. We prove in Section \ref{sec:classification} that the representation category $\Rep(\HH_T)$ is given by $\mathcal{C}_f(K<G)$. We thus obtain the consecutive quotients of discrete quantum groups $\widehat{\GG_T} \recht \widehat{\HH_T} \recht \Gamma_T$. In particular, we also provide numerous examples of bicrossed products with property~(T). Note here that $\Gamma_T$ is a highly proper quotient of $\widehat{\GG_T}$ and that there is no obvious way to obtain the discrete quantum group $\widehat{\GG_T}$ or its von Neumann algebra from $\Gamma_T$ or any other discrete group. It is rather so that several quite different mathematical objects, including tensor categories, buildings, planar algebras, quantum groups and discrete groups, can be constructed from the same intricate combinatorics provided by triangle presentations.

In general, whenever $K$ is a compact open subgroup of a locally compact group $G$ and whenever $\Gamma < G$ is a discrete group that complements $K$ in the sense that $G = \Gamma \, K$ and $\Gamma \cap K = \{e\}$, we construct a fiber functor on the rigid C$^*$-tensor category $\mathcal{C}_f(K<G)$ whose corresponding compact quantum group is precisely the bicrossed product of $\Gamma$ and $K$. This leads us to the intriguing open question whether all fiber functors on $\mathcal{C}_f(K<G)$ are given by such complementing discrete subgroups $\Gamma < G$~?

In the final Section \ref{sec:diagram-calculus}, we initiate the classification up to monoidal equivalence of the rigid $C^*$-tensor categories $\Rep(\GG_T)$, which are likely to only depend on the size of the set $F$. We provide in particular a diagrammatic description of the intertwiner spaces $(u^{\otimes n},u^{\otimes m})$, where $u$ is the fundamental representation of $\GG_T$. The question whether $\Rep(\GG_T)$ only depends on the size of the set $F$ then reduces to a very intricate, purely combinatorial open problem that we present in Remark \ref{rem:indication-mon-equiv}. We also provide numerical evidence pointing towards a positive answer to this problem, at least when $\Delta_T$ is a classical Bruhat--Tits $\tilde{A}_2$-building.

\section{Preliminaries}

\subsection{\boldmath Rigid \texorpdfstring{$C^*$}{C*}-tensor categories}\label{sec.tensor-category}

In this section, we briefly recall a number of facts about rigid $C^*$-tensor categories.
A comprehensive introduction to the subject matter may be found in \cite[Chapter~2]{neshveyev-tuset}.
Throughout, all tensor categories are assumed to be strict and essentially small.
Additionally, we always assume that $C^*$-tensor categories are closed under direct sums and passage to subobjects, unless otherwise specified.

A \textit{rigid $C^*$-tensor category} is a $C^*$-tensor category $\mathcal{C}$ in which every object $\alpha\in\mathcal{C}$ has a conjugate $\bar{\alpha}\in\mathcal{C}$.
We denote the tensor product of $\alpha,\beta\in\mathcal{C}$ by $\alpha\beta$, and the unit object of $\mathcal{C}$ by $\varepsilon$.
Given $\alpha,\beta\in\mathcal{C}$, the normed vector space of morphisms $\alpha\to\beta$ will be denoted by $(\beta,\alpha)$.
The rigidity assumption implies that all these morphism spaces are finite-dimensional.
In particular, every object $\alpha\in\mathcal{C}$ splits into a finite sum of irreducibles, and the \textit{multiplicity} of an irreducible object $\gamma$ in $\alpha$ is defined by $\mult(\gamma,\alpha)=\dim_{\CC}(\gamma,\alpha)$.
The set of isomorphism classes of irreducible objects in $\mathcal{C}$ will be denoted by $\Irr(\mathcal{C})$.
We always work with some fixed choice of representatives for all $\alpha\in\Irr(\mathcal{C})$, and we do not distinguish between irreducible objects and their isomorphism classes.
For all $\alpha\in\Irr(\mathcal{C})$, we identify $(\alpha,\alpha)$ with $\CC$.

Given an object $\alpha\in\mathcal{C}$, the rigidity assumption says that $\alpha$ admits a conjugate object $\bar{\alpha}\in\mathcal{C}$.
This means that there exists a pair of maps $s_\alpha\in(\alpha\bar{\alpha},\varepsilon)$ and $t_\alpha\in(\bar{\alpha}\alpha,\varepsilon)$ such that
\[
(t_\alpha^*\otimes 1)(1\otimes s_\alpha)=1 \qquad \text{and}\qquad (s_\alpha^*\otimes 1)(1\otimes t_\alpha)=1;.
\]
The pair $(s_\alpha,t_\alpha)$ is referred to as a \textit{solution to the conjugate equations}.
A solution to the conjugate equations is \textit{standard} when it additionally satisfies
\begin{equation}\label{eqn:categorical-trace}
s_\alpha^*(T\otimes 1)s_\alpha=t_\alpha^*(1\otimes T)t_\alpha \quad\text{for all}\;\; T \in (\al,\al) \;.
\end{equation}
Conjugate objects and standard solutions are unique up to unitary conjugacy and satisfy various naturality properties with respect to direct sums and tensor products in $\mathcal{C}$ (see \cite[\S~2.2]{neshveyev-tuset}).
The positive quantity $d(\alpha)=s_\alpha^*s_\alpha=t_\alpha^*t_\alpha$ is called the \textit{categorical dimension} of $\alpha$.
Throughout the article, we always consider fixed standard solutions for all irreducibles, and extend those to arbitrary objects by naturality.

Given a standard solution $(s_\alpha,t_\alpha)$, the identity \eqref{eqn:categorical-trace} determines a faithful positive tracial functional on $(\alpha,\alpha)$, given by
\[
\Tr_\alpha(T)=s_\alpha^*(T\otimes 1)s_\alpha=t_\alpha^*(1\otimes T)t_\alpha\;.
\]
This map is referred to as the \textit{categorical trace} on $(\alpha,\alpha)$, and does not depend on the choice of standard solution.
Typically $\Tr_\alpha$ is not normalized, since $\Tr_\alpha(1)=d(\alpha)$.
For any two objects $\alpha,\beta\in\mathcal{C}$, the intertwiner space $(\alpha,\beta)$ comes with a natural inner product induced by the categorical traces, which is given by
\begin{equation}\label{eqn:intertw-innprod}
\langle T, S\rangle=\Tr_\alpha(TS^*)=\Tr_\beta(S^*T)\;.
\end{equation}
The notation $\onb(\alpha,\beta)$ will always refer to an orthonormal basis of $(\alpha,\beta)$ with respect to this inner product.

Standard solutions also induce \textit{Frobenius reciprocity maps} between morphism spaces, given by
\[
\begin{aligned}
            &(\alpha\beta,\gamma) \to (\alpha,\gamma\overline{\beta}): T\mapsto (1\otimes s_\beta^*)(T\otimes 1) \; ,\\
            &(\alpha\beta,\gamma)\to (\beta,\overline{\alpha}\gamma): T\mapsto (t_\alpha^*\otimes 1)(1\otimes T) \; ,
\end{aligned}
\]
where $\alpha,\beta,\gamma\in\Irr(\mathcal{C})$.
These maps are unitary with respect to the inner product defined in~\eqref{eqn:intertw-innprod}.

\subsection{Compact quantum groups}

A \textit{compact quantum group} $\GG$ in the sense of Woronowicz \cite{woron-cqg} consists of a unital $C^*$-algebra $A$, along with a unital $*$-homomorphism $\Delta:A\to A\otimes_{\mathrm{min}} A$ satisfying
\begin{itemize}
        \item co-associativity: $(\Delta\otimes\id)\Delta=(\id\otimes\Delta)\Delta$
        \item density conditions: $\Delta(A)(1\otimes A)$ and $\Delta(A)(A\otimes 1)$ are dense in $A\otimes_{\mathrm{min}} A$.
\end{itemize}
Every compact quantum group $\GG$ admits a unique \textit{Haar state} $h$ on $A$, which is characterized by $(\id\otimes h)\Delta(b)=(h\otimes\id)\Delta(b)=h(b)1$.
If the Haar state is a trace, we say that $\GG$ is of \textit{Kac type}.

A finite-dimensional representation of $\GG$ on $\CC^n$ is a unitary $u\in M_n(\CC)\otimes A$ such that $\Delta(u_{ij})=\sum^n_{k=1} u_{ik}\otimes u_{kj}$.
Given representations $u$ on $\CC^n$ and $v$ on $\CC^m$, a linear map $T:\CC^n\to \CC^m$ \textit{intertwines} $u$ and $v$ if $v(T\otimes 1)=(T\otimes 1)u$.
The category $\Rep(\GG)$ of finite-dimensional representations of $\GG$ with intertwiners as morphisms is a rigid $C^*$-tensor category.

A \textit{fiber functor} on a rigid $C^*$-tensor category $\mathcal{C}$ is a faithful unitary monoidal functor $\mathbf{F}:\mathcal{C}\to\mathrm{Hilb}_f$, where $\mathrm{Hilb}_f$ denotes the rigid $C^*$-tensor category of finite-dimensional Hilbert spaces.
Note that, if $\GG$ is a compact quantum group, the forgetful functor from $\Rep(\GG)$ to $\mathrm{Hilb}_f$ is a fiber functor.
Conversely, Woronowicz' Tannaka--Krein duality theorem \cite{woron-tkdual} shows that one can recover the quantum group from its representation category and a fiber functor.
In fact, this is how many examples of compact quantum groups are constructed.
Given a fiber functor $\mathbf{F}$ on a rigid $C^*$-tensor category $\mathcal{C}$, we say that $\mathbf{F}$ is of \textit{Kac type} whenever $d(\alpha)=\dim_\CC(\mathbf{F}(\alpha))$ for all $\alpha\in\mathcal{C}$.
A quantum group is of Kac type if and only if the associated fiber functor is of Kac type.

\subsection{The tube algebra}\label{sec:tube-algebra}

The representation theory of a rigid $C^*$-tensor category can be conveniently viewed through the lens of a certain class of modules over its \textit{tube algebra}, see \cite{ghosh-jones,psv-cohom}.
The tube algebra construction was first introduced by Ocneanu \cite{ocneanu-chirality} for fusion categories, i.e.\@ $C^*$-tensor categories with finitely many isomorphism classes of irreducible objects.

Let $\mathcal{C}$ be a rigid $C^*$-tensor category. The underlying vector space of the tube algebra $\mathcal{A}$ is defined by the algebraic direct sum
\[
\mathcal{A}=\bigoplus_{i,j,\alpha\in\Irr(\mathcal{C})} (i\alpha,\alpha j)\;.
\]
When writing $V\in (i\alpha,\alpha j)$, we implicitly mean that $V$ is an element of the summand indexed by $i,j,\alpha$ in the tube algebra.

Given $i,j\in\Irr(\mathcal{C})$ and $\alpha\in\mathcal{C}$ arbitrary, we can still associate an element of $\mathcal{A}$ to a morphism $V\in (i\alpha,\alpha j)$ via the map
\begin{equation}\label{eqn:morphism-to-tube-alg}
V\mapsto \sum_{\gamma\in\Irr(\mathcal{C})} d(\gamma)\sum_{W\in\onb(\alpha,\gamma)} (1\otimes W^*)V(W\otimes 1)\;.
\end{equation}
Note that the expression on the right is independent of all basis choices.
Some care should be taken when viewing morphisms as tube algebra elements in this way, since the map in \eqref{eqn:morphism-to-tube-alg} is typically not injective.

The $*$-algebra structure on $\mathcal{A}$ is then defined by
\begin{align*}
        V\cdot W &= \delta_{j,j'}(V\otimes 1)(1 \otimes W)\in (i\alpha\beta,\alpha\beta k) \; ,\\
        V^\# &= (t_\alpha^*\otimes 1 \otimes 1)(1 \otimes V^*\otimes 1)(1 \otimes 1 \otimes s_\alpha)\in (j\overline{\alpha},\overline{\alpha} i) \; ,
\end{align*}
for $V\in(i\alpha,\alpha j)$ and $W\in (j'\beta, \beta k)$.
To avoid confusion with composition of morphisms in $\mathcal{C}$, we systematically denote the multiplication of $V$ and $W$ in $\mathcal{A}$ by $V\cdot W$.

For any irreducible object $i\in\Irr(\mathcal{C})$, the identity in $(i\varepsilon,\varepsilon i)$ becomes a self-adjoint idempotent when viewed as an element of $\mathcal{A}$, which will be denoted by $p_i$.
Observe that $p_i\cdot V\cdot p_j=\delta_{ik}\delta_{jk'}V$ when $V\in (k\alpha,\alpha k')$.
The tube algebra is unital if and only if the number of irreducible objects is finite, but the elements $p_i$ always serve as local units, in the sense that every $V\in\mathcal{A}$ is supported under a finite sum of $p_i$'s.

The corner $p_\varepsilon\cdot \mathcal{A}\cdot p_\varepsilon$ can be identified with the \textit{fusion $*$-algebra} $\CC[\mathcal{C}]$.
This is the direct analogue of the group $*$-algebra of a discrete group.
As a vector space, $\CC[\mathcal{C}]$ is spanned by $\Irr(\mathcal{C})$.
The involution is defined by conjugation, and multiplication by
\[
        \alpha\cdot\beta = \sum_{\gamma\in\Irr(\mathcal{C})} \mult(\gamma,\alpha\beta)\gamma\;.
\]
By mapping $\alpha\in\Irr(\mathcal{C})$ to the identity intertwiner in $(\varepsilon \alpha,\alpha\varepsilon)$, we can consider $\alpha$ as an element of $p_\varepsilon\cdot\mathcal{A}\cdot p_\varepsilon$.

One can consider the representation theory of a rigid $C^*$-tensor category $\mathcal{C}$ from several points of view \cite{pv-repr-subfactors,neshveyev-yamashita,ghosh-jones,psv-cohom}.
Here, we consider \textit{nondegenerate right Hilbert modules} over the tube algebra $\mathcal{A}$.
By \cite[Lemma~3.9]{psv-cohom}, any linear *-representation of $\mathcal{A}$ on an inner product space is automatically by bounded operators, and $\|V\|\leq d(\alpha)$ for all $V\in (i\alpha,\alpha j)$, where $i,j,\alpha\in\Irr(\mathcal{C})$.
A right *-representation $\mathcal{A}$ on a Hilbert space $\mathcal{H}$ is said to be \textit{nondegenerate} when the subspaces $\mathcal{H}\cdot p_i, i\in\Irr(\mathcal{C})$ span a dense subspace of $\mathcal{H}$.

In this framework, the \textit{trivial representation} is defined via the \textit{counit} $\counit:\mathcal{A}\to\CC$, which is given by
\[
        \counit(p_i)=\delta_{i,\varepsilon}\qquad \text{ and } \qquad \counit(\alpha) = d(\alpha)\;,
\]
for all $i,\alpha\in\Irr(\mathcal{C})$.

\subsection{\boldmath Property~(T) for rigid \texorpdfstring{$C^*$}{C*}-tensor categories}

There are several equivalent definitions of property~(T), see \cite{popa-symmenv, pv-repr-subfactors, neshveyev-yamashita, ghosh-jones}.
The quantitative characterization stated below, which is essentially \cite[Proposition~4.22(ii)]{neshveyev-yamashita}, is the most suitable one for our purposes.

\begin{definition}\label{def:invar-property-T}
Let $\mathcal{C}$ be a rigid $C^*$-tensor category with tube algebra $\mathcal{A}$.
Consider a nondegenerate right Hilbert $\mathcal{A}$-module $\mathcal{H}$.
A vector $\xi\in\mathcal{H}$ is \textit{invariant} if $\xi\cdot 1_\alpha = d(\alpha)\xi$ for all $\xi\in\mathcal{H}$.
A net $(\xi_i)_{i\in I}$ in $\mathcal{H}$ is said to be \textit{almost invariant} if $\xi_i\cdot 1_\alpha-d(\alpha)\xi_i\to 0$ for all $\alpha\in\Irr(\mathcal{C})$.

Given a finite set $F\subset\Irr(\mathcal{C})$ and $\varepsilon>0$, a vector $\xi\in\mathcal{H}$ is $(F,\varepsilon)$\textit{-invariant} whenever
\[
\|\xi\cdot 1_\alpha-d(\alpha)\xi\|<d(\alpha)\varepsilon\|\xi\|.
\]
for all $\alpha\in F$.
There exists an almost invariant net of unit vectors in $\mathcal{H}$ if and only if $\mathcal{H}$ admits $(F,\varepsilon)$-invariant vectors for all finite sets $F\subset\Irr(\mathcal{C})$ and all $\varepsilon>0$.

A \textit{Kazhdan pair} for $\mathcal{C}$ is a pair $(F,\varepsilon)$ with $F\subset\Irr(\mathcal{C})$ a finite subset and $\varepsilon>0$ such that any nondegenerate right Hilbert $\mathcal{A}$-module containing an $(F,\varepsilon)$-invariant vector must contain a nonzero invariant vector.
If $\mathcal{C}$ admits a Kazhdan pair, we say that $\mathcal{C}$ has \textit{property~(T)}.
Equivalently, $\mathcal{C}$ has property~(T) if and only if a nondegenerate right Hilbert module with almost invariant vectors admits a nonzero invariant vector.
\end{definition}

\begin{remark}
Given that we work with modules over the full tube algebra $\mathcal{A}$, it might seem unusual to define invariance and almost-invariance in terms of the corner $p_\varepsilon\cdot\mathcal{A}\cdot p_\varepsilon$.
It is however routine to verify the following equivalences for all nondegenerate right Hilbert $\mathcal{A}$-modules $\mathcal{H}$:
\begin{enumerate}[(i)]
            \item A vector $\xi\in\mathcal{H}$ is invariant in the sense of Definition~\ref{def:invar-property-T} if and only if $\xi\cdot V=\counit(V)\xi$ for all $V\in\mathcal{A}$.
            \item A net of unit vectors $(\xi_i)_{i\in I}$ is almost invariant in the sense of Definition~\ref{def:invar-property-T} if and only if $\xi_i\cdot V-\counit(V)\xi_i\to 0$ for all $V\in\mathcal{A}$.
\end{enumerate}
\end{remark}

\subsection{\boldmath Bruhat--Tits \texorpdfstring{$\tilde{A}_2$}{A\_2}-building of a local field}\label{sec:bruhat-tits}

Let $\KK$ be a non-Archimedean local field with valuation $\nu : \KK^\times \recht \ZZ$. We briefly recall the construction of the associated Bruhat--Tits $\tilde{A}_2$-building and refer to \cite[Chapter~9]{ronan} for further details.

Denote by $\cO = \{x \in \KK \mid \nu(x) \geq 0\}$ the ring of integers and fix a generator $\varpi$ for the unique maximal ideal $\{x \in \KK \mid \nu(x) \geq 1\}$ of $\cO$. A lattice in $\KK^3$ is a finitely generated $\cO$-submodule of $\KK^3$ that generates $\KK^3$ as a $\KK$-vector space. Two lattices $L$ and $L'$ are called equivalent if $L=tL'$ for some $t\in\KK^\times$. The vertices of the building Bruhat--Tits $\tilde{A}_2$-building $\Delta_\KK$ are by definition the equivalence classes of lattices. There is an edge between two vertices $x,x' \in \Delta_\KK$ if and only if $x,x'$ admit representatives $L,L'$ such that $\varpi L\subsetneq L'\subsetneq L$.

The natural action of the projective linear group $\PGL(3,\KK)$ on the vertices of $\Delta_\KK$ is transitive and the stabilizer of the equivalence class of $\mathcal{O}^3$ is $\PGL(3,\mathcal{O})$.

\subsection{\boldmath Triangle presentations and groups acting on \texorpdfstring{$\tilde{A}_2$}{A\_2}-buildings}\label{sec.triangle-pres}

In \cite{cmsz1}, groups acting simply transitively on the vertices of a Euclidean building of type $\tilde{A}_2$ were described combinatorially through the following notion of a \emph{triangle presentation}.

\begin{definition}[\cite{cmsz1}]\label{def:triangle-pres}
A \textit{triangle presentation} over a finite \textit{base set} $F$ is a subset $T\subset F\times F\times F$ satisfying the following conditions.
\begin{enumerate}[(i)]
\item $T$ is invariant under cyclic permutations: if $(i,j,k)\in T$, then $(k,i,j)\in T$.
\item For all pairs $(i,j)\in F\times F$, there is at most one $k\in F$ such that $(i,j,k)\in T$. If such a $k$ exists, we denote this relationship by $i\to j$, and we say that $i$ is a \textit{predecessor} of $j$, or that $j$ is a \textit{successor} of $i$.
\item Any two distinct points in $F$ have exactly one common predecessor and exactly one common successor.
\item There exists an integer $q \geq 2$ such that every point in $F$ has exactly $q+1$ predecessors and $q+1$ successors. We call $q$ the \emph{order} of the triangle presentation.
\end{enumerate}
\end{definition}

As the notation suggests, we think of $T$ as a directed graph with vertex set $F$ and edges labeled by vertices. The statement $(i,j,k)\in T$ then translates to the existence of an edge from $i$ to $j$ with label $k$. Since every vertex in $T$ has exactly $q+1$ successors, we find that the number of elements in $T$ is always $(q+1)N$, where $N$ is the cardinality of $F$.

In \cite{cmsz1}, to every triangle presentation $T$ is associated the countable group
\begin{equation}\label{eqn:triangle-group}
\Gamma_T=\langle a_x, x\in F\mid a_i a_j a_k = e \;\;\text{for all}\;\; (i,j,k)\in T\rangle\;.
\end{equation}

By \cite[Theorem 3.4]{cmsz1}, the Cayley graph of $\Gamma_T$ with respect to the generating set $S=\{a_i,a_i^{-1}\mid i\in F\}$ is the $1$-skeleton of a thick $\tilde{A}_2$-building $\Delta_T$. In \cite{cms}, it is shown that the groups $\Gamma_T$ have property~(T). Moreover, it is proved in \cite[\S~5]{zuk-propT} that these groups satisfy the spectral gap criterion of \cite{zuk-propT} with respect to the generating set $S$. This relied on a result of \cite{feit-higman} about spectra of incidence graphs of projective planes, which are associated with a triangle presentation in the following way.

\begin{remark}\label{rem:triag-induced-projective-plane}
The conditions satisfied by a triangle presentation $T$ of order $q$ over the base set $F$ implicitly endow $F$ with the structure of a projective plane $P$ of order $q$.
More precisely, the incidence relation is the following: for all $i,j\in F$ the point with label $j$ lies on the line with label $i$ whenever $i\to j$ in $T$.
Condition~(iii) then states that any two lines intersect exactly once and that any two points lie on a common line. Condition~(iv) says that the plane is of order $q$.
In particular, this implies that $|F|=q^2+q+1$. In all known cases, the order of a projective plane is a prime power.

The definition of a triangle presentation formally makes sense as well for $q=1$, but then the resulting groups $\Gamma_T$ no longer have property~(T) and the set $F$ no longer becomes a projective plane.
\end{remark}

Triangle presentations of orders $2$ and $3$ were fully classified in \cite{cmsz2}. We content ourselves with giving two examples in Table~\ref{tbl:triang-pres-examples} and refer to \cite{cmsz2} for a complete list.

\begin{table}[ht]
    \centering
    \begin{tabular}{c c c|c c c}
        & Example 1 & & & Example 2 &  \\
        \hline
        (1,0,4) & (0,4,1) & (4,1,0) & (0,1,3) & (1,3,0) & (3,0,1)\\
        (2,0,1) & (0,1,2) & (1,2,0) & (1,2,4) & (2,4,1) & (4,1,2)\\
        (1,6,3) & (6,3,1) & (3,1,6) & (2,3,5) & (3,5,2) & (5,2,3)\\
        (2,5,6) & (5,6,2) & (6,2,5) & (3,4,6) & (4,6,3) & (6,3,4)\\
        (4,3,5) & (3,5,4) & (5,4,3) & (4,5,0) & (5,0,4) & (0,4,5)\\
        (3,6,5) & (6,5,3) & (5,3,6) & (5,6,1) & (6,1,5) & (1,5,6)\\
        (4,0,2) & (0,2,4) & (2,4,0) & (6,0,2) & (0,2,6) & (2,6,0)\\
    \end{tabular}
    \caption{Two examples of triangle presentations of order $2$ on $F=\{0,\ldots,6\}$, corresponding to groups~$B.1$ and~$A.1$ in \cite{cmsz2}, respectively.}
    \label{tbl:triang-pres-examples}
\end{table}

For certain triangle presentations $T$, the building $\Delta_T$ is the Bruhat--Tits building of type $\tilde{A}_2$ associated with a non-Archimedean local field $\KK$. By \cite[Theorem 1]{cmsz2}, this holds for all triangle presentations of order $2$. In particular by \cite[Theorem 1]{cmsz2}, the examples in Table \ref{tbl:triang-pres-examples} give rise to the $\tilde{A}_2$-buildings of the local field $\KK = \QQ_2$ of $2$-adic numbers (Example 1) and the local field $\KK = \FF_2((X))$ of formal Laurent series over the field $\FF_2$ with two elements (Example 2). However by \cite[Theorem 2]{cmsz2}, there are numerous triangle presentations $T$ of order $3$ whose associated building $\Delta_T$ is exotic. Finally, in \cite[\S~4]{cmsz1} it is shown that for any prime power $q$, there exists a triangle presentation $T$ of order $q$ such that $\Delta_T$ is Bruhat--Tits.

\section{A \.Zuk-type spectral gap criterion for property~(T)}

\subsection{Statement and notation}

Let $\mathcal{C}$ be a rigid $C^*$-tensor category, and $S\subset\Irr(\mathcal{C})$.
We say that $S$ is \textit{symmetric} if $\alpha\in S$ implies that $\bar{\alpha}\in S$, and that $S$ \textit{generates} $\mathcal{C}$ if every irreducible object in $\mathcal{C}$ appears in some tensor product of elements of $S$.
We always assume that generating sets do not contain the unit object $\varepsilon$.

Before stating the criterion, we first introduce an analogue of the link of a generating set of a discrete group.
Given a finite symmetric set $S$, we define the weight function $\nu$ given by
\begin{equation}\label{eqn:genset-weight-function}
        \nu: S\to (0,+\infty): \alpha\mapsto \sum_{\beta,\gamma\in S} \frac{\mult(\gamma,\bar{\beta}\alpha)}{d(\gamma)}.
\end{equation}
The \textit{vertex space} of the link of $S$ is the direct sum
\begin{equation}\label{eqn:vertex-space}
\mathcal{V}=\bigoplus_{\substack{\alpha\in S \\ i\in\Irr(\mathcal{C})}} (i\alpha,\alpha)\; .
\end{equation}
    Note that this direct sum is finite, since it runs over the irreducible components of $\alpha\bar{\alpha}$ for all $\alpha\in S$.
    Given a vector $\xi\in\mathcal{V}$, we denote the $(i\alpha,\alpha)$-summand by $\xi_{i,\alpha}$.
    Equip $\mathcal{V}$ with a reweighted inner product given by
    \[
        \langle \xi,\eta\rangle_\nu = \sum_{\substack{\alpha\in S\\ i\in\Irr(\mathcal{C})}} \nu(\alpha)\Tr_\alpha(\eta_{i,\alpha}^*\xi_{i,\alpha})\; .
    \]
    To avoid notational inconsistencies, the notation $\onb(i\alpha,\alpha)$ will always refer to an orthonormal basis with respect to the inner product induced by the categorical trace.
    The \textit{edge space} of the link of $S$ is the direct sum
    \[
        \mathcal{E} =\bigoplus_{\substack{\alpha,\beta,\gamma\in S \\ i\in\Irr(\mathcal{C})}} (\bar{\beta}i\alpha,\gamma)\otimes (\gamma,\bar{\beta}\alpha)
    \]
    equipped with the usual inner product that arises from the categorical trace.
    Given a vector $\xi\in\mathcal{E}$ we denote the component in $(\bar{\beta}i\alpha,\gamma)\otimes (\gamma,\bar{\beta}\alpha)$ by $\xi_{\alpha,\beta,\gamma,i}$.

    Next, we introduce \textit{source and target maps} $\partial_\ell, \partial_r:\mathcal{V}\to\mathcal{E}$.
    For $V\in (i\alpha,\alpha)$, define $\tilde{V}\in (\bar{\alpha}i,\bar{\alpha})$ by Frobenius reciprocity:
    \begin{equation*}
        \tilde{V} = V^{\#*} = (1 \ot 1 \otimes s_\alpha^*)(1\otimes V\otimes 1)(t_\alpha\otimes 1)\;.
    \end{equation*}
    The source and target maps are then defined by
    \begin{equation*}
        \begin{aligned}
        ((\partial_\ell)(\xi))_{\alpha,\beta,\gamma,i} &= \sum_{W\in\onb(\bar{\beta}\alpha,\gamma)} (1\otimes\xi_{i,\alpha})W\otimes W^*\;,\\
        ((\partial_r)(\xi))_{\alpha,\beta,\gamma,i} &= \sum_{W\in\onb(\bar{\beta}\alpha,\gamma)} (\tilde{\xi}_{i,\beta}\otimes 1)W\otimes W^*\;.
        \end{aligned}
    \end{equation*}
    Observe that these expressions are independent of the choices of orthonormal bases.

    The \textit{Laplacian} of $S$ is defined by
    \begin{equation}
    \label{eqn:laplacian-def}
    \Delta=\frac{1}{2} (\partial_\ell-\partial_r)^*(\partial_\ell-\partial_r)\;,
    \end{equation}
    where the adjoint is taken with respect to the $\langle\cdot,\cdot\rangle_\nu$-inner product.

    If $\mathcal{C}$ is a discrete group, $\mathcal{V}$ (resp.\@ $\mathcal{E}$) canonically identifies with the space of functions on the vertices (resp.\@ edges) of the link of $S$.
    Under this identification, $\partial_\ell$ and $\partial_r$ correspond to the duals of the source and target maps of the graph.
    It is also easy to see that the constant functions on the link of $S$ are mapped to scalar multiples of the canonical vector $C\in\mathcal{V}$ given by
    \begin{equation}
        \label{eqn:constant-vector}
        C_{i,\alpha}=\delta_{i,\varepsilon}1_\alpha.
    \end{equation}
    Denote the projection onto the span of $C$ by $P_C$.
    We will denote the square of the norm of $C$ by $\kappa$.
    Note that
    \begin{equation*}
        \kappa = \sum_{\alpha\in S} \nu(\alpha) d(\alpha).
    \end{equation*}

    \begin{remark}
    \label{rem:principal-graph-part}
    The Laplacian $\Delta$ restricts to a linear transformation $\Delta_i$ of $\bigoplus_{\alpha\in S} (i\alpha,\alpha)$ for every $i\in\Irr(\mathcal{C})$.
    When $i=\varepsilon$, the direct summand $\Delta_\eps$ can be naturally identified with the graph Laplacian of a weighted graph: the graph has vertex set $S$ and edge weights $w(\alpha,\beta)=\sum_{\gamma\in S} d(\gamma)^{-1}\mult(\gamma,\bar{\beta}\alpha)$. In particular, $\Delta_\eps$ only depends on the fusion rules and the dimension function of $\cC$. In the spectral gap criterion for property~(T) in Theorem \ref{thm:spectral-gap-criterion}, we also need to take into account the direct summands $\Delta_i$ with $i \neq \eps$, which depends on more information than the fusion rules and the dimension function.
    \end{remark}

    The goal of this section is to prove the following spectral gap criterion for property~(T). This is a tensor category version of \.Zuk's criterion for discrete groups proved in \cite[Theorem 1]{zuk-propT}.
    \begin{theorem}
    \label{thm:spectral-gap-criterion}
    Let $\mathcal{C}$ be a rigid $C^*$-tensor category generated by a finite symmetric set $S \subset \Irr(\cC) \setminus \{\eps\}$.
    Define as above the operator $\Delta$ with respect to $S$, and denote its smallest nonzero eigenvalue by $\lambda_1$.
    If zero is a simple eigenvalue of $\Delta$ and $\lambda_1 > 1/2$, then $\mathcal{C}$ has property~(T).
    More precisely, $(S,2-\lambda_1^{-1})$ is a Kazhdan pair for $\mathcal{C}$.
    \end{theorem}
    Observe that, given some generating set $S\subset\Irr(\mathcal{C})$, the matrix $\Delta$ can be computed from finite-dimensional data:
    \begin{itemize}
        \item a basis of the endomorphism space of $\alpha\bar{\alpha}$ for all $\alpha\in S$;
        \item a basis of $(\alpha\beta,\gamma)$ for all $\alpha,\beta,\gamma\in S$.
    \end{itemize}
    Nevertheless, finding such bases in concrete examples can be far from trivial.
    This is already the case for the examples that we consider in Section~\ref{sec:new-examples}.

    As we shall see in the next section, Theorem~\ref{thm:spectral-gap-criterion} allows for a very short and elementary proof of property~(T) for $\Rep(\SU_\mu(3))$, where $\mu \in (-1,1) \setminus \{0\}$. This provides an alternative to the approach of \cite{pv-repr-subfactors,arano-suqn}.

\subsection{Proof of Theorem~\ref{thm:spectral-gap-criterion}}

The following lemma generalizes \cite[Lemma 12.1.8]{brown-ozawa}, characterizing property~(T) in terms of spectral gap of a certain averaging operator.
The proof is almost identical.

\begin{lemma}
    \label{thm:spectral-gap-lemma}
    Let $\mathcal{C}$ be a rigid $C^*$-tensor category generated by a finite symmetric set of irreducibles $S$.
    Denote the tube algebra of $\mathcal{C}$ by $\mathcal{A}$.
    Let $\nu:S\to (0,+\infty)$ be a symmetric function (i.e.\@ $\nu(\alpha)=\nu(\bar{\alpha})$ for all $\alpha\in S$), and define
    \begin{equation}
        \kappa =\sum_{\alpha\in S} \nu(\alpha)d(\alpha)\quad\text{and}\quad h = \frac{1}{\kappa}\sum_{\alpha\in S} \nu(\alpha) 1_\alpha \quad \in p_\varepsilon\cdot\mathcal{A}\cdot p_\varepsilon \numberthis\label{eqn:avg-op-definition}.
    \end{equation}
    Then the following are equivalent:
    \begin{enumerate}[(i)]
        \item $\mathcal{C}$ has property~(T).
        \item In a universal nondegenerate right Hilbert $\mathcal{A}$-module, $1$ is isolated in the spectrum of $h$. More concretely, if $\sigma(h)\subset [-1,1-\varepsilon]\cup \{1\}$ with $\varepsilon>0$, then $(S,\varepsilon)$ is a Kazhdan pair for $\mathcal{C}$.
    \end{enumerate}
    \begin{proof}
        Let $\mathcal{H}$ be a universal nondegenerate right Hilbert $\mathcal{A}$-module.
        By uniform convexity, a vector $\xi\in\mathcal{H}$ is invariant if and only if $\xi\cdot h=\xi$, and a sequence of vectors~$(\xi_n)_{n\in\NN}$ in $\mathcal{H}$ is almost invariant if and only if $\xi_n\cdot h - \xi_n\to 0$. So we immediately get that $\cC$ has property~(T) if and only if $1$ is isolated in the spectrum of $h$.

        Assume that $\sigma(h)\subset[-1,1-\varepsilon]\cup\{1\}$ holds for the universal nondegenerate right Hilbert $\cA$-module and hence, for an arbitrary nondegenerate right Hilbert $\cA$-module $\cH$.
        Let $P$ be the projection onto the fixed vectors of $h$.
        By the spectral mapping theorem, $\varepsilon(1-P)\leq 1-h$.
        Therefore, any $(S,\eps)$-invariant vector $\xi \in \cH$ satisfies
        \begin{align*}
            \|\xi-P\xi\|^2 &= \langle (1-P)\xi,\xi\rangle \leq \frac{1}{\varepsilon}\langle \xi\cdot (1-h), \xi\rangle
            =\frac{1}{\varepsilon\kappa}\sum_{\alpha\in S}\nu(\alpha)\langle d(\alpha)\xi-\xi\cdot 1_\alpha,\xi\rangle\\
            &\leq \frac{1}{\varepsilon\kappa}\|\xi\|\sum_{\alpha\in S} \nu(\alpha)\|d(\alpha)\xi-\xi\cdot 1_\alpha\|
            < \frac{1}{\kappa}\|\xi\|^2\sum_{\alpha\in S} \nu(\alpha)d(\alpha)=\|\xi\|^2\; .
        \end{align*}
        It follows that $P\xi\neq 0$ so that $\mathcal{H}$ admits nonzero invariant vectors. This means that $(S,\eps)$ is a Kazhdan pair.
    \end{proof}
\end{lemma}

    We are now ready to prove Theorem~\ref{thm:spectral-gap-criterion}.

\begin{proof}[Proof of Theorem~\ref{thm:spectral-gap-criterion}]
    Take $\nu$ as in \eqref{eqn:genset-weight-function}, and define $h$ using the formula \eqref{eqn:avg-op-definition}.
    Define $\Delta$ as in~\eqref{eqn:laplacian-def}.
    By Lemma \ref{thm:spectral-gap-lemma}, it is sufficient to prove that the spectral gap property of $\Delta$ implies that $h\in p_\varepsilon\cdot\mathcal{A}\cdot p_\varepsilon$ has spectral gap.
    To relate $\Delta$ and $h$, we define an amplification operation on $\mathcal{A}$-modules.
    Given a nondegenerate right Hilbert $\mathcal{A}$-module $\mathcal{H}$ and a vector $\xi\in\mathcal{H}\cdot p_\varepsilon$, we can amplify $\xi$ to a vector $\Amp(\xi)\in\mathcal{V}\otimes\mathcal{H}$ by setting
    \begin{equation*}
    \Amp(\xi)_{i,\alpha} = \sum_{W\in\onb(i\alpha,\alpha)} \sqrt{d(i)} W\otimes\xi\cdot W^\#.
    \end{equation*}
    In particular, if $\xi$ is invariant we find that $\Amp(\xi)=C\otimes\xi$, with $C$ as in \eqref{eqn:constant-vector}.
    With $h$ defined as in Lemma~\ref{thm:spectral-gap-lemma}, observe that
    \begin{equation*}
    \begin{split}
        & \langle (P_C\otimes 1) \Amp(\xi),\Amp(\xi)\rangle = \sum_{\alpha,\beta\in S} \frac{1}{d(\alpha\beta)}\langle P_C\alpha,P_C\beta\rangle_\nu\langle\xi\cdot 1_{\bar{\alpha}},\xi\cdot 1_{\bar{\beta}}\rangle\\
        &=\sum_{\alpha,\beta\in S} \frac{1}{d(\alpha\beta)}\left\langle \frac{\nu(\alpha)d(\alpha)}{\kappa}C, \frac{\nu(\beta)d(\beta)}{\kappa}C\right\rangle_\nu \langle \xi\cdot 1_{\bar{\alpha}},\xi\cdot 1_{\bar{\beta}}\rangle
        =\sum_{\alpha,\beta\in S} \frac{\nu(\alpha)\nu(\beta)}{\kappa}\langle \xi\cdot 1_{\bar{\alpha}},\xi\cdot 1_{\bar{\beta}}\rangle \\ & =\kappa\|\xi\cdot h\|^2\;,
    \end{split}
    \end{equation*}
    for all $\xi\in\mathcal{H}\cdot p_\varepsilon$.
    By \cite[Lemma~3.9]{psv-cohom}, we also have that  $\|\Amp(\xi)\|^2=\kappa\|\xi\|^2$, so
    \begin{equation}
    \label{eqn:constant-complement-identity}
    \frac{1}{\kappa}\langle ((1-P_C)\otimes 1)\Amp(\xi),\Amp(\xi)\rangle = \|\xi\|^2-\|\xi\cdot h\|^2=\langle \xi\cdot (1-h^2),\xi\rangle\ .
    \end{equation}
    Under the assumptions of Theorem~\ref{thm:spectral-gap-criterion}, the left-hand side of this identity can be estimated in terms of $\langle (\Delta\otimes 1)\Amp(\xi),\Amp(\xi)\rangle$.
    To aid in the computation of these matrix coefficients, we introduce an auxiliary family of sesquilinear forms $M^{\bar{\beta}\alpha}_\gamma$ on $\mathcal{V}$, indexed by $\alpha,\beta,\gamma\in S$ and given by
    \[
        M^{\bar{\beta}\alpha}_\gamma(\xi,\eta)=\sum_{\substack{i\in\Irr(\mathcal{C})\\U\in\onb(\bar{\beta}\alpha,\gamma)}} \Tr_{\gamma}(U^*(\eta_{i,\beta}^\#\otimes 1)(1\otimes\xi_{i,\alpha})U)
    \]
    These play the role of the Markov operator $M=1-\Delta$, since an easy computation yields that
    \begin{equation}
    \label{eqn:laplacian-markov-matrix}
    \langle (\partial_\ell-\partial_r)(\xi),(\partial_\ell-\partial_r)(\eta)\rangle = 2\langle\xi,\eta\rangle_\nu - \sum_{\alpha,\beta,\gamma\in S}\left[M^{\bar{\beta}\alpha}_\gamma(\xi,\eta)+\overline{M^{\bar{\alpha}\beta}_\gamma(\eta,\xi)}\right]
    \end{equation}
    for all $\xi,\eta\in\mathcal{V}$.
    Before breaking down how the Laplacian interacts with the amplification procedure described above, we observe that
    \begin{align*}
    \sum_{W\in\onb(i\beta,\beta)}\overline{M^{\bar{\alpha}\beta}_\gamma(W,V)} W
    &=\sum_{U\in\onb(\bar{\alpha}\beta,\gamma)} (s_\alpha^*\otimes 1)(1\otimes\tilde{V}\otimes 1)(1\otimes UU^*)(s_\alpha\otimes 1) \\
    &=d(\gamma)^{-1}(1\otimes s_\alpha^*\otimes 1)(V\otimes P^{\bar{\alpha}\beta}_\gamma)(s_\alpha\otimes 1)\;.
    \end{align*}
    for all $\alpha,\beta\in S$, $i\in\Irr(\mathcal{C})$ and $V\in (i\alpha,\alpha)$.
    Here, $P^{\bar{\alpha}\beta}_\gamma\in(\bar{\alpha}\beta,\bar{\alpha}\beta)$ denotes the orthogonal projection of $\bar{\alpha}\beta$ onto the largest subobject isomorphic to a direct sum of copies of $\gamma$.

    This leads to
    \begin{equation*}
    \begin{split}
        & \sum_{\substack{i\in\Irr(\mathcal{C}) \\ V\in\onb(i\alpha,\alpha) \\ W\in\onb(i\beta,\beta)}} \ d(i)\overline{M^{\bar{\alpha}\beta}_\gamma(W,V)} (V^\#\otimes 1)(1\otimes W) \\
        &=d(\gamma)^{-1} \sum_{\substack{i\in\Irr(\mathcal{C}) \\ Z\in\onb(\alpha\bar{\alpha},i)}} d(i)(t_\alpha^*\otimes 1 \ot 1)(1\otimes ZZ^*\otimes 1) (1 \ot 1 \otimes P^{\bar{\alpha}\beta}_\gamma)(1\otimes s_\alpha\otimes 1) = d(\gamma)^{-1} P^{\bar{\alpha}\beta}_\gamma\;,
    \end{split}
    \end{equation*}
    for $\alpha,\beta,\gamma\in S$. In $p_\varepsilon\cdot\mathcal{A}\cdot p_\varepsilon$, we obtain the following identity.
    \[
        \sum_{\substack{i\in\Irr(\mathcal{C}) \\ V\in\onb(i\alpha,\alpha) \\ W\in\onb(i\beta,\beta)}} d(i)\overline{M^{\bar{\alpha}\beta}_\gamma(W,V)} V^\#\cdot W= \frac{\mult(\gamma,\bar{\alpha}\beta)}{d(\gamma)}1_\gamma \;.
    \]
    By symmetry, we also get that the following holds in $p_\varepsilon\cdot\mathcal{A}\cdot p_\varepsilon$.
    \[
        \sum_{\substack{i\in\Irr(\mathcal{C}) \\ V\in\onb(i\alpha,\alpha) \\ W\in\onb(i\beta,\beta)}} d(i)M^{\bar{\beta}\alpha}_\gamma(V,W) V^\#\cdot W =\frac{\mult(\bar{\gamma},\bar{\alpha}\beta)}{d(\gamma)}1_{\bar{\gamma}}\;.
    \]
    Then, for $\xi,\eta\in\mathcal{H}$, we apply \eqref{eqn:laplacian-markov-matrix} to obtain
    \begin{align*}
    \langle (\Delta\otimes 1)\Amp(\xi),\Amp(\eta)\rangle &=\kappa\langle\xi,\eta\rangle - \sum_{\alpha,\beta,\gamma\in S} \frac{\mult(\gamma,\bar{\beta}\alpha)}{d(\gamma)}\langle\xi\cdot 1_\gamma,\eta\rangle\\
    &=\kappa\langle\xi,\eta\rangle - \left\langle\xi\cdot \left(\sum_{\gamma\in S} \nu(\gamma)1_\gamma\right) ,\eta\right\rangle=\kappa\left(\langle\xi\cdot (1-h), \eta\rangle\right).
    \end{align*}
    In particular, when $\xi=\eta$, we find that
    \begin{equation}
    \label{eqn:laplacian-amplification-identity}
    \frac{1}{\kappa}\langle (\Delta\otimes 1)\Amp(\xi),\Amp(\xi)\rangle = \langle \xi\cdot (1-h),\xi\rangle\; .
    \end{equation}

    We now show that $\sigma(h)\subset [-1,\lambda_1^{-1}-1] \cup \{1\}$.
    Since zero is a simple eigenvalue of $\Delta$, we have that $\Delta\geq\lambda_1(1-P_C)$.
    By \eqref{eqn:constant-complement-identity} and \eqref{eqn:laplacian-amplification-identity}, this implies that
    \begin{equation*}
        \lambda_1\langle \xi\cdot (1-h^2),\xi\rangle = \frac{1}{\kappa}\langle (\lambda_1(1-P_C)\otimes 1)\Amp(\xi),\Amp(\xi)\rangle \leq \frac{1}{\kappa}\langle (\Delta\otimes 1)\Amp(\xi),\Amp(\xi)\rangle =\langle \xi\cdot(1-h),\xi\rangle
    \end{equation*}
    for all $\xi\in\mathcal{H}$. By the spectral mapping theorem, $\sigma(h)\setminus\{1\}\subset [-1,\lambda_1^{-1}-1]$.
    The theorem now follows by Lemma~\ref{thm:spectral-gap-lemma}, since $\lambda_1^{-1}-1<1$.
    \end{proof}

\section{\boldmath A new proof of property~(T) for \texorpdfstring{$\Rep(\SU_\mu(3))$}{Rep(SU\_q(3))}}

    \begin{proposition}\label{prop:apply-to-su-q-3}
    Fix $\mu\in (-1,1)\setminus\{0\}$ and denote the fundamental representation of $\SU_\mu(3)$ by $u$.
    Then $(\{u,\bar{u}\},\varepsilon_\mu)$ is a Kazhdan pair for $\Rep(\SU_\mu(3))$, with the Kazhdan constant $\varepsilon_\mu$ given by
    \[
        \varepsilon_\mu= \frac{|\mu|+|\mu|^{-1}-2}{|\mu|+|\mu|^{-1}-1} \;.
    \]
    \end{proposition}
    \begin{proof}
    We first recall from \cite{woron-tkdual} a number of facts about $\SU_\mu(3)$.
    The fundamental representation $u$ of $\SU_\mu(3)$ is three-dimensional. Up to scalar multiples,
    there is a unique invariant vector in $u^{\otimes 3}$, which we denote by $E$.
    \begin{equation*}
        E=e_{123}-\mu e_{132}-\mu e_{213}+\mu^2 e_{231}+\mu^2e_{312} - \mu^3 e_{321}\in (\CC^3)^{\otimes 3},
    \end{equation*}
    where $e_{ijk}$ denotes the vector $e_i\otimes e_j\otimes e_k$.
    By definition, all intertwiners between tensor powers of the fundamental representation $u$ are given by taking linear combinations of products operators of the form $1^{\otimes a}\otimes E\otimes 1^{\otimes b}$ and their adjoints. It is easy to see that $E$ satisfies the relations
    \begin{equation*}
        \begin{aligned}
            (E^*\otimes 1)(1\otimes E)&=(1\otimes E^*)(E\otimes 1)=\mu^2(1+\mu^2)1,\\
            (E^*\otimes 1 \ot 1)(1 \ot 1 \otimes E)&=(1 \ot 1 \otimes E^*)(E\otimes 1 \ot 1).
        \end{aligned}
    \end{equation*}
    The map
    $$P=\mu^{-2}(1+\mu^2)^{-1}(E^*\otimes 1 \ot 1)(1 \ot 1 \otimes E) \in (uu,uu)$$
    is a projection, and its image can be identified with the conjugate representation $\bar{u}$ of $u$.
    We denote the isometric inclusion of $\bar{u}$ into $uu$ by~$Y$.
    The following two maps then solve the conjugate equations for $u$ and $\bar{u}$:
    \begin{equation*}
        s_u: \varepsilon\to u\bar{u}: \mu^{-1}(1+\mu^2)^{-1/2} (1\otimes Y^*)E\;, \quad
        t_u: \varepsilon\to \bar{u}u: \mu^{-1}(1+\mu^2)^{-1/2} (Y^*\otimes 1)E\;.
    \end{equation*}
    One also checks that
    \begin{equation}
    \label{eqn:suq3-hecke-relation}
        (1\otimes P)(P\otimes 1)(1\otimes P) = \left(\frac{1}{\mu+\mu^{-1}}\right)^2((1\otimes P)-L)+L\,
    \end{equation}
    where $L$ is the projection onto the span of $E$.

    We now verify that $\SU_\mu(3)$ satisfies the conditions of Theorem~\ref{thm:spectral-gap-criterion} when $\mu\neq 1$, by computing the Laplacian with respect to the generating set $S=\{u,\bar{u}\}$.
    Since $\SU_\mu(3)$ has the same fusion rules as $\SU(3)$, the irreducible representations of $\SU_\mu(3)$ are naturally labeled by Young diagrams with columns of at most height two.
    We use the symbol $u_{ab}$ to denote the irreducible representation of $\SU_\mu(3)$ associated with the Young diagram with~$a$ columns of height one, and~$b$ columns of height two.
    This means that $u=u_{10}$ and $\bar{u}=u_{01}$.
    The tensor products of the form $\alpha\beta$ with $\alpha,\beta\in\{u,\bar{u}\}$ decompose into irreducibles as follows:
    \begin{equation*}
        uu =u_{20}\oplus \bar{u}\;,\qquad
        \bar{u}\bar{u} =u_{02}\oplus u\;,\qquad
        u\bar{u} =\bar{u}u=\varepsilon\oplus u_{11}\;.
    \end{equation*}
    Hence, the vertex space $\mathcal{V}=(u,u)\oplus (\bar{u},\bar{u})\oplus (u_{11}u,u)\oplus (u_{11}\bar{u},\bar{u})$ is four-dimensional.
    We also get that $\nu(u)=\nu(\bar{u})=d(u)^{-1}=(\mu^2+1+\mu^{-2})^{-1}$. So choosing a basis of isometries gives an orthonormal basis of $\mathcal{V}$ with respect to the $\langle \cdot,\cdot\rangle_\nu$-inner product.

    To proceed, we therefore have to compute explicit isometries in $(u_{11}u,u)$ and $(u_{11}\bar{u},\bar{u})$.
    Denote the inclusion map of $u_{11}$ into $u\bar{u}$ by $I_{u\bar{u}}$.
    Using \eqref{eqn:suq3-hecke-relation}, we find that
    \[
        I_{\bar{u}u} = (\mu^{-1}+\mu)(Y^*\otimes 1)(1\otimes Y)I_{u\bar{u}}
    \]
    is an isometry from $u_{11}$ into $\bar{u}u$.
    By Frobenius reciprocity, the following maps are also isometries:
    \begin{align*}
        Z_{u}&=\sqrt{\frac{d(u)}{d(u)^2-1}}(I_{u\bar{u}}^*\otimes 1)(1\otimes t_u)\in (u_{11}u,u),\\
        Z_{\bar{u}}&=\sqrt{\frac{d(u)}{d(u)^2-1}}(I_{\bar{u}u}^*\otimes 1)(1\otimes s_u)\in (u_{11}\bar{u},\bar{u}).
    \end{align*}
    With respect to the basis $1_u,1_{\bar{u}},Z_{u},Z_{\bar{u}}$ of $\mathcal{V}$ the Laplacian $\Delta$ is given by the matrix
    \[
        \Delta=\left(\begin{matrix}
        1 & -1 & 0 & 0 \\
        -1 & 1 & 0 & 0\\
        0 & 0 & 1 & -\zeta\\
        0 & 0 & -\zeta & 1
        \end{matrix}\right)\;.
    \]
    with $\zeta=d(u)^{-1}\Tr_{\bar{u}}(Y^*(Z_{\bar{u}}^\#\otimes 1)(1\otimes Z_u)Y)$.
    To conclude the computation, note that
    \[
        (Z_{\bar{u}}^\#\otimes 1)(1\otimes Z_{u})=\frac{1}{d(u)^2-1}(\mu^{-1}+\mu)(d(u)YY^*- 1\otimes 1)\;,
    \]
    from which we obtain that
    \[
        \zeta=\frac{\mu^{-1}+\mu}{d(u)+1}=\frac{\mu^{-1}+\mu}{\mu^{-2}+2+\mu^2}=\frac{1}{\mu^{-1}+\mu}\; .
    \]
    Hence, the spectrum of $\Delta$ is given by
    \[
        \sigma(\Delta)=\left\{0,1-\frac{1}{\mu^{-1}+\mu}, 1+\frac{1}{\mu^{-1}+\mu},2\right\}\; .
    \]
    So, the smallest positive eigenvalue $\lambda_1$ of $\Delta$ equals $1 - (|\mu|+|\mu|^{-1})^{-1}$. Since $0 < |\mu| < 1$, we get that $\lambda_1 > 1/2$. The result then follows from Theorem \ref{thm:spectral-gap-criterion}.
    \end{proof}

\section{Discrete quantum groups with property~(T)}\label{sec:new-examples}


Recall from the introduction that property~(T) for discrete quantum groups $\hat{\GG}$ was introduced in \cite{fima-prop-T}. It is related in the following way to property~(T) of the representation category $\Rep(\GG)$. For arbitrary compact quantum groups $\GG$, by \cite[Proposition 6.3]{pv-repr-subfactors}, property~(T) of $\Rep(\GG)$ is equivalent with the central property~(T) for $\hat{\GG}$ introduced in \cite{arano-suqn}. The ordinary property~(T) of $\hat{\GG}$ is a strictly stronger notion, because by \cite[Proposition 3.2]{fima-prop-T}, it forces $\GG$ to be of Kac type. For Kac type discrete quantum groups $\hat{\GG}$, by \cite[Proposition 8.7]{arano-drinfeld-unitary-duals}, central property~(T) is equivalent with property~(T) and thus also equivalent with property~(T) of $\Rep(\GG)$. By \cite[Theorem 6.8]{arano-delaat-wahl}, in the Kac case, we moreover have equivalence with property~(T) of the tracial von Neumann algebra $L^\infty(\GG)$.

\subsection{Discrete quantum groups given by a triangle presentation}

Recall from Section \ref{sec.triangle-pres} the notion of a triangle presentation $T$ and the associated group $\Gamma_T$ whose Cayley graph $\Delta_T$ is an $\tilde{A}_2$-building.

\begin{definition}\label{def:G-T}
Let $T$ be a triangle presentation over a base set $F$ in the sense of Definition \ref{def:triangle-pres}. We define $C(\GG_T)$ as the universal unital $C^*$-algebra generated by elements $u_{ij}, i,j\in F$ such that $u=(u_{ij})_{i,j\in F}$ is a unitary element of $B(\CC^F)\otimes C(\GG_T)$ and
\begin{equation}\label{eqn:woron-twisted-unimodularity}
        \sum_{(a,b,c)\in T} u_{ia} u_{jb} u_{kc} = \begin{cases}
            1 & \;\;\text{if}\;\; (i,j,k) \in T \; ,\\
            0 & \;\;\text{if}\;\; (i,j,k) \notin T \; ,
        \end{cases}
\end{equation}
for all $i,j,k\in F$. We uniquely define the unital $*$-homomorphism $\Delta : C(\GG_T) \recht C(\GG_T) \ot C(\GG_T)$ so that $\Delta(u_{ij}) = \sum_k u_{ik} \ot u_{kj}$.
\end{definition}

In Lemma \ref{lem:compact-quantum} below, we prove that the $*$-homomorphism $\Delta$ is well defined and turns $\GG_T$ into a compact quantum group of Kac type. Then, $u$ is a unitary representation of $\GG_T$, which we call the \textit{fundamental representation}.

The following is the main result of this section. The proofs of parts (a) and (b) are given in Sections~\ref{sec:triag-pres-laplacian} and~\ref{sec:triag-pres-building}, respectively.

\begin{theorem}\label{thm:repgt-main-theorem}
Let $T$ be a triangle presentation of order $q\geq 2$.
Denote the fundamental representation of $\GG_T$ by $u$.
Then the following hold:
\begin{enumerate}[(a)]
        \item If $\dim(u\bar{u},u\bar{u})=3$, then $\Rep(\GG_T)$ and $\widehat{\GG_T}$ have property~(T).
        \item If the building $\Delta_T$ is isomorphic to the Bruhat--Tits $\tilde{A}_2$-building of a non-Archimedean commutative local field $\KK$, then $\dim(u\bar{u},u\bar{u})=3$ and $\Rep(\GG_T)$ admits irreducible representations of arbitrarily high dimension.
\end{enumerate}
\end{theorem}


\begin{remark}
Note that under the hypothesis of point (a) in Theorem \ref{thm:repgt-main-theorem}, the tracial von Neumann algebra $L^\infty(\GG_T)$ has property~(T) as a von Neumann algebra. The structure of this von Neumann algebra is highly not understood. In particular, this von Neumann algebra seems to have no obvious connection to a group von Neumann algebra or a crossed product von Neumann algebra.
\end{remark}

\begin{remark}\label{rem:planar-algebra}
Triangle presentations $T$ of order $q \geq 2$ satisfying the hypotheses of Theorem \ref{thm:repgt-main-theorem} also gives rise to standard $\lambda$-lattices in the sense of \cite{popa-lambda-lattices} and to subfactor planar algebras in the sense of \cite{jones:planar} that have property~(T) in the sense of \cite{popa-symmenv}. Denote by $\cC = \Rep(\GG_T)$ the representation category of the compact quantum group $\GG_T$, with fundamental representation $u \in \cC$. A concrete description of the intertwiner spaces is given by Remark \ref{rem:Tannaka-point-of-view} below. Define $\cC_{00} = \cC_{11} \subset \cC$ as the subcategory of representations that appear in $u^{\ot n}$ with $n \equiv 0 \mod 3$. Define $\cC_{01} \subset \cC$ as those representations that appear in  $u^{\ot n}$ with $n \equiv 1 \mod 3$, and define $\cC_{10} \subset \cC$ as those that appear in $u^{\ot n}$ with $n \equiv 2 \mod 3$.

By definition, $u \in \cC_{01}$. Also by definition,
$$\cC_{00} \ot \cC_{01} \ot \cC_{11} \subset \cC_{01} \;\; , \quad \overline{\cC_{01}} = \cC_{10} \;\; , \quad \cC_{01} \ot \cC_{10} \subset \cC_{00} \; .$$
By Lemma \ref{lem:generating}, the tensor powers of $u \ot \bar{u}$ generate $\cC_{00}$. So we have defined a rigid $C^*$-$2$-category $\cC$ together with a generator $u \in \cC_{01}$.
As explained in e.g.\ \cite[Remark 2.1]{arano-vaes}, every such $C^*$-$2$-category can be viewed as a standard $\lambda$-lattice or as a subfactor planar algebra. More precisely, there exists an extremal finite index subfactor $N \subset M$ such that $\cC_{ij}$ can be identified with the category of resp.\ all $N$-$N$-bimodules, $N$-$M$-bimodules, $M$-$N$-bimodules and $M$-$M$-bimodules appearing in the Jones tower of $N \subset M$, in such a way that $u \in \cC_{01}$ corresponds to the generating $N$-$M$-bimodule $L^2(M)$.

Under the hypotheses of Theorem \ref{thm:repgt-main-theorem} and combining \cite[Propositions 5.2 and 5.7]{pv-repr-subfactors} with Lemma \ref{lem:generating} below, this standard $\lambda$-lattice has property~(T) in the sense of \cite[Section 9]{popa-symmenv}.
\end{remark}

\begin{lemma}\label{lem:compact-quantum}
The $*$-homomorphism $\Delta$ in Definition \ref{def:G-T} is well defined and turns $\GG_T$ into a compact quantum group of Kac type. The element $\bar{u} \in B(\CC^F) \otimes C(\GG_T)$ defined by $(\bar{u})_{ij} = u^*_{ij}$ is unitary and $\GG_T$ is of Kac type.
\end{lemma}

\begin{proof}
Define the vector $E \in (\CC^F)^{\ot 3}$ given by
\begin{equation}\label{eqn:triangle-invariant-vector}
E=\sum_{(a,b,c)\in T} e_a \ot e_b \ot e_c \; ,
\end{equation}
which we also view as an operator from $\CC$ to $(\CC^F)^{\ot 3}$. Using the tensor leg numbering notation for the element $U = (u_{ij})_{i,j \in F}$ of $B(\CC^F) \ot C(\GG_T)$, the relation in \eqref{eqn:woron-twisted-unimodularity} becomes
\begin{equation}\label{eqn:with-U}
U_{14} \, U_{24} \, U_{34} \, (E \ot 1) = E \ot 1 \; .
\end{equation}
To prove that there is a well defined unital $*$-homomorphism $\Delta : C(\GG_T) \recht C(\GG_T) \ot C(\GG_T)$ as in Definition \ref{def:G-T}, satisfying $(\id \ot \Delta)(U) = U_{12} \, U_{13}$, we have to prove that
$$\bigl( U_{14} \, U_{15} \bigr) \, \bigl( U_{24} \, U_{25} \bigr) \, \bigl( U_{34} \, U_{35} \bigr) \, (E \ot 1 \ot 1) = E \ot 1 \ot 1 \; .$$
This follows immediately because the left hand side equals
$$\bigl(U_{14} \, U_{24} \, U_{34}\bigr) \; \bigl(U_{15} \, U_{25} \, U_{35}\bigr) \, (E \ot 1 \ot 1) =  U_{14} \, U_{24} \, U_{34} \, (E \ot 1 \ot 1) = E \ot 1 \ot 1\; .$$

Define the element $\bar{U} \in B(\CC^N) \ot C(\GG_T)$ given by $(\bar{U})_{ij} = u_{ij}^*$. Define for every $i \in F$, the vector $\xi_i \in \CC^F \ot \CC^F$ given by
\begin{equation}\label{eqn:vector-xi}
\xi_i = \sum_{j,k, (i,j,k) \in T} e_j \ot e_k \; .
\end{equation}
Define the isometry $Y : \CC^N \recht \CC^N \ot \CC^N$ given by $Y e_i = (q+1)^{-1/2} \xi_i$.

It follows from \eqref{eqn:with-U} that $U_{24} \, U_{34} \, (E \ot 1) = U_{14}^* (E \ot 1)$. This equation can be rewritten as
\begin{equation}\label{eqn:first-Ubar}
U_{13} \, U_{23} \, (Y \ot 1) = (Y \ot 1) \, \bar{U} \; .
\end{equation}
Since $U$ is unitary, it also follows from \eqref{eqn:with-U} that $(E^* \ot 1) \, U_{14} \, U_{24} \, U_{34} = E^* \ot 1$, so that
$$(E^* \ot 1) \, U_{14} \, U_{24} = (E^* \ot 1) \, U_{34}^* \; .$$
Since $T$ is invariant under cyclic permutations, we conclude that $(Y^* \ot 1) \, U_{13} \, U_{23} = \bar{U} \, (Y^* \ot 1)$. Because $Y$ is an isometry and $U$ is unitary, it follows in combination with \eqref{eqn:first-Ubar} that $\bar{U}$ is unitary. So, $\Delta$ turns $C(\GG_T)$ into a compact quantum group of Kac type.
\end{proof}

\begin{remark}
Note that $\GG_T$ is nontrivial: if $\pi$ is an arbitrary unitary representation of the discrete group $\Gamma_T$ on some Hilbert space $\mathcal{H}$, we can realize the defining relations of~$C(\GG_T)$ inside $\mathcal{B}(\mathcal{H})$ by mapping $u_{ij}$ to $\delta_{ij}\pi(a_i)$. So, the discrete group $\Gamma_T$ is a quotient of the discrete quantum group $\widehat{\GG_T}$. This already shows that $C(\GG_T)$ is noncommutative.
\end{remark}

\begin{remark}\label{rem:Tannaka-point-of-view}
Let $T$ be a triangle presentation with associated compact quantum group $\GG_T$ with fundamental representation $u$. Since $\bar{u}$ is a subrepresentation of $u u$, all irreducible representations of $\GG_T$ appear as a subrepresentation of a tensor power of $u$. By the Tannaka--Krein theorem of Woronowicz, the morphism spaces $(u^{\ot n},u^{\ot m})$ are linearly spanned by products of intertwiners of the form $1^{\ot a} \ot E \ot 1^{\ot b}$ and $1^{\ot c} \ot E^* \ot 1^{\ot d}$, where $E$ is defined by \eqref{eqn:triangle-invariant-vector}. In particular, $(u^{\ot n},u^{\ot m}) = \{0\}$ when $n-m$ is not a multiple of $3$.
\end{remark}

\subsection{Proof of Theorem \ref{thm:repgt-main-theorem}, part (a)}\label{sec:triag-pres-laplacian}

The aim of this section is to use Theorem~\ref{thm:spectral-gap-criterion} to prove part (a) of Theorem~\ref{thm:repgt-main-theorem}. So we first analyze the endomorphisms of a few low-order tensor powers of the fundamental representation of $\GG_T$.

Fix a finite set $F$ and a triangle presentation $T \subset F \times F \times F$ of order $q \geq 2$. Consider the associated compact quantum group $\GG_T$ with fundamental representation $u$ as in Definition \ref{def:G-T}. Write $N = 1 + q + q^2$ and identify $F = \{1,\ldots,N\}$. Denote the standard basis of $\CC^N$ by $e_1,\ldots,e_N$. Given $I=(i_1,\ldots,i_k)\in F^k$, we denote by $e_I$ the vector $e_{i_1}\otimes\cdots \otimes e_{i_k}$ in $(\CC^N)^{\otimes k}$. We also often denote the multi-index $(i_1,\ldots,i_n)$ as the word $i_1\cdots i_n$.

Define the vector $E \in (\CC^N)^{\ot 3}$ given by \eqref{eqn:triangle-invariant-vector} and the isometry
$$Y : \CC^N \recht \CC^N \ot \CC^N : Y e_i= (q+1)^{-1/2} \sum_{j,k, (i,j,k) \in T} e_{jk} \; .$$
By Lemma \ref{lem:compact-quantum}, the operator $\bar{u} \in B(\CC^N) \ot C(\CC_T)$ given by $(\bar{u})_{ij} = u_{ij}^*$ is a unitary representation of $\GG_T$, which is by construction the contragredient of $u$. Also by construction, $Y \in (uu,\bar{u})$ and
\begin{equation*}
t_u: \varepsilon\to \bar{u} \otimes u: \frac{1}{\sqrt{q+1}} (Y^*\otimes 1) E \; , \quad
s_u: \varepsilon\to u\otimes \bar{u}: \frac{1}{\sqrt{q+1}} (1\otimes Y^*) E \; ,
\end{equation*}
solve the conjugate equations for $u$.


Define the intertwiner $P \in (uu,uu)$ given by
\begin{equation*}
P=\frac{1}{q+1}(E^*\otimes 1 \ot 1)(1 \ot 1 \otimes E)=\frac{1}{q+1}(1 \ot 1 \ot E^*)(E\otimes 1 \ot 1) = YY^* \;.
\end{equation*}
Then, $P$ is the orthogonal projection of $\CC^N \ot \CC^N$ onto the $N$-dimensional subrepresentation of $uu$ spanned by the vectors $\xi_i = Y e_i$, $i \in F$.

The following lemma shows that the behavior of the fundamental representation of $\GG_T$ differs from that of $\SU_\mu(3)$ because the tensor products $uu$, $u\bar{u}$ and $\bar{u}u$ have more endomorphisms. The proof of the lemma follows by a straightforward computation.

\begin{lemma}\label{thm:basic-intertwiners}
Define the projections
\begin{itemize}
    \item $R$ of $(\CC^N)^{\otimes 2}$ onto the subspace spanned by the vectors $\{e_{ij}\mid i,j\in F, i\to j\text{ in } T\}$,
    \item $Q$ of $(\CC^N)^{\otimes 3}$ onto the subspace spanned by the vectors $\{e_i \ot \xi_i \mid i \in F\}$.
\end{itemize}
Then
    \begin{align}
    \label{eqn:heckeish-relation}
    (1\otimes P)(P\otimes 1)(1\otimes P)&=\frac{1}{(q+1)^3} \left((q+1)(1\otimes P) + (q^2-1)Q + EE^*\right)\;,\\
    \label{eqn:q-r-frobenius}
    R&=(E^*\otimes 1 \ot 1)(1\otimes Q\otimes 1)(1 \ot 1 \otimes E)\;.
    \end{align}
    In particular, $R\in (uu,uu)$ and $Q\in (uuu,uuu)$.

Also, $(Y^* \ot 1) (1 \ot Y)$ is an invertible element of $(\bar{u} u , u \bar{u})$, so that $u\bar{u}\cong\bar{u}u$ and
\begin{equation}
    \label{eqn:intw-space-dimensions}
        \dim(uu,uu)=\dim(u\bar{u},\bar{u}u)=\dim(u\bar{u},u\bar{u})\geq 3\;.
    \end{equation}
\end{lemma}

%
%
%

The hypothesis of part (a) of Theorem~\ref{thm:repgt-main-theorem} asserts that the inequality in \eqref{eqn:intw-space-dimensions} is an equality, while part (b) of that theorem gives a sufficient condition for this to be the case.

\begin{proof}[Proof of Theorem~\ref{thm:repgt-main-theorem}, part (a)]
The representation $u \bar{u}$ on the Hilbert space $\CC^N \ot \CC^N$ contains the trivial representation $\eps$, spanned by the invariant vector $\sum_i e_i \ot e_i$. By Lemma \ref{thm:basic-intertwiners}, the vectors $e_i \ot e_i, i \in F$, span a subrepresentation $\zeta$ of $u \bar{u}$. We define the representation $\al$ as the orthogonal complement of the trivial representation $\eps$ inside $\zeta$. So $\al$ is given by restricting $u \bar{u}$ to the invariant subspace of vectors of the form $\sum_i \lambda_i e_i \ot e_i$ with $\sum_i \lambda_i = 0$. Finally, we denote by $\be$ the orthogonal complement of $\eps \oplus \al$ inside $u \bar{u}$. By construction, $d(\alpha)=N-1$ and $d(\beta)=N(N-1)$. The assumption that $\dim (u \bar{u},u\bar{u}) = 3$ then ensures that $u$ is irreducible and that $u \bar{u}$ splits as the direct sum of the inequivalent irreducible representations $u \bar{u} = \varepsilon\oplus\alpha\oplus\beta$.

By Remark \ref{rem:Tannaka-point-of-view}, there are no nonzero intertwiners between $u\bar{u}$ and $uu$, so that none of $\alpha$, $\beta$, $\eps$ embed into $uu$.

We now compute the Laplacian of $\Rep(\GG_T)$ with respect to the generating set $S=\{u,\bar{u}\}$. The vertex space $\mathcal{V}$ is 6-dimensional and given by
\[
\mathcal{V}=(u,u)\oplus (\bar{u},\bar{u})\oplus (\alpha u, u)\oplus (\alpha \bar{u}, \bar{u}) \oplus (\beta u, u) \oplus (\beta\bar{u},\bar{u})\;.
\]
The weight function $\nu$ is given by $\nu(u)=\nu(\bar{u})=d(u)^{-1}=N^{-1}$. As with our computation for $\SU_\mu(3)$, this means that we can work with a basis of isometries. Denote the inclusion of $\alpha$ (resp.\@ $\beta$) in $u\bar{u}$ by $I_{u\bar{u}}^\alpha$ (resp.\@ $I_{u\bar{u}}^\beta$). It follows from Lemma \ref{thm:basic-intertwiners} that
\begin{align*}
I_{\bar{u}u}^\alpha &= \frac{q+1}{\sqrt{q}} (Y^*\otimes 1)(1\otimes Y) I_{u\bar{u}}^\alpha\quad \in (\bar{u}u,\alpha)\;,\\
I_{\bar{u}u}^\beta &= (q+1)(Y^*\otimes 1)(1\otimes Y)I_{u\bar{u}}^\beta\quad \in (\bar{u}u,\beta)
\end{align*}
are isometric intertwiners. Frobenius reciprocity then yields the following list of isometries:
\begin{align*}
Z_{u}^\alpha&=\sqrt{\frac{N}{N-1}}  ((I_{u\bar{u}}^\alpha)^*\otimes 1)(1\otimes t_u)\quad\in (\alpha u,u)\;, \\
Z_{\bar{u}}^\alpha &=\sqrt{\frac{N}{N-1}} ((I_{\bar{u}u}^\alpha)^*\otimes 1)(1\otimes s_u)\quad \in (\alpha \bar{u},\bar{u})\;,\\
Z_{u}^\beta&=\sqrt{\frac{1}{N-1}} ((I_{u\bar{u}}^\beta)^*\otimes 1)(1\otimes t_u)\quad \in (\beta u,u)\;,\\
Z_{\bar{u}}^\beta &=\sqrt{\frac{1}{N-1}} ((I_{\bar{u}u}^\beta)^*\otimes 1)(1\otimes s_u)\quad \in (\beta \bar{u},\bar{u})\;.
\end{align*}
With respect to the basis $1_u,1_{\bar{u}}, Z_{u}^\alpha, Z_{\bar{u}}^\alpha, Z_{u}^\beta, Z_{\bar{u}}^\beta$ of $\mathcal{V}$, the Laplacian is given by the block matrix
\[
\Delta = \left(\begin{matrix} 1 & -1 \\ -1 & 1\end{matrix}\right) \oplus \left(\begin{matrix} 1 & -\lambda_\alpha \\ -\lambda_\alpha & 1\end{matrix}\right)\oplus \left(\begin{matrix} 1 & -\lambda_\beta \\ -\lambda_\beta & 1\end{matrix}\right)
\]
where
\begin{align*}
\lambda_\alpha &= N^{-1}\Tr_{\bar{u}}(Y^*((Z_{\bar{u}}^\alpha)^\#\otimes 1)(1\otimes Z_{u}^\alpha)Y)\;,\\
\lambda_\beta &= N^{-1}\Tr_{\bar{u}}(Y^*((Z_{\bar{u}}^\beta)^\#\otimes 1)(1\otimes Z_{u}^\beta)Y)\;.
\end{align*}
Calculating $\lambda_\alpha$ and $\lambda_\beta$ is now rather straightforward.
Making use of Lemma \ref{thm:basic-intertwiners}, one readily verifies that
\begin{align*}
((Z_{\bar{u}}^\alpha)^\#\otimes 1)(1\otimes Z_u^\alpha) &= \frac{1}{N-1}\frac{1}{\sqrt{q}} (NR-(q+1)(1 \ot 1))\;,\\
((Z_{\bar{u}}^\beta)^\#\otimes 1)(1\otimes Z_u^\beta) &= \frac{1}{N-1}\left((q+1)P-R\right)\;.
\end{align*}
After taking traces, we obtain
\begin{align*}
\lambda_\alpha &= \frac{1}{N-1}\frac{1}{\sqrt{q}}(N-(q+1)) = \frac{\sqrt{q}}{q+1}\;,\\
\lambda_\beta &= \frac{1}{N-1}((q+1)-1)=\frac{1}{q+1}\;.
\end{align*}
We conclude that the spectrum of $\Delta$ is given by
\[
\sigma(\Delta)=\left\{0\;,\;1-\frac{\sqrt{q}}{q+1}\;,\;1-\frac{1}{q+1}\;,\;1+\frac{1}{q+1}\;,\;1+\frac{\sqrt{q}}{q+1}\;,\;2\right\}\;.
\]
Since $\sqrt{q}/(q+1)<1/2$ for all $q\geq 2$, we conclude that the smallest eigenvalue of $\Delta$ is strictly larger than $1/2$.
By Theorem~\ref{thm:spectral-gap-criterion}, we conclude that $(\{u,\bar{u}\}, \varepsilon_q)$ is a Kazhdan pair for $\Rep(\GG_T)$, where
\[
\varepsilon_q=\frac{(\sqrt{q}-1)^2}{q-\sqrt{q}+1}\;.
\]
Since $\GG_T$ is of Kac type, it follows that also $\widehat{\GG_T}$ has property~(T) (see beginning of Section \ref{sec:new-examples} for references).
\end{proof}

\begin{lemma}\label{lem:generating}
Let $T$ be a triangle presentation of order $q \geq 2$ with associated compact quantum group $\GG_T$ with fundamental representation $u$.

There is a well defined map $\tau : \Irr(\GG_T) \recht \ZZ/3\ZZ$ satisfying $\tau(\al) = k$ if $\al \subset u^{\ot n}$ and $n \equiv k \mod 3$. This map satisfies $\tau(\gamma) = \tau(\al)+\tau(\be)$ when $\gamma \subset \al \ot \be$ and $\tau(\bar{\al}) = - \tau(\al)$, for all $\al,\be,\gamma \in \Irr(\GG_T)$.

Defining the full $C^*$-tensor subcategory $\cC_0 \subset \Rep(\GG_T)$ as the kernel of $\tau$, i.e.\ as the category of representations that are a direct sum of $\al \in \Irr(\GG_T)$ with $\tau(\al) = 0$, we get that $\cC_0$ is generated by $u \ot \overline{u}$.
\end{lemma}
\begin{proof}
It follows from Remark \ref{rem:Tannaka-point-of-view} that $\tau$ is well defined. By construction, $\tau(\gamma) = \tau(\al)+\tau(\be)$ when $\gamma \subset \al \ot \be$. Since $\bar{u} \subset u \ot u$, it also follows that $\tau(\bar{\al}) = - \tau(\al)$, for all $\al,\be,\gamma \in \Irr(\GG_T)$.

By definition, $\cC_0$ is generated by $u^{\ot 3}$. So it remains to prove that $u^{\ot 3}$ is a subrepresentation of a tensor power of $u \ot \bar{u}$. By Lemma \ref{thm:basic-intertwiners}, $\bar{u} \ot u \cong u \ot \bar{u}$. It follows that $(u \ot \bar{u})^{\ot 3} \cong u^{\ot 3} \ot {\bar{u}}^{\ot 3}$. Since the trivial representation is contained in ${\bar{u}}^{\ot 3}$, the lemma is proved.
\end{proof}

\subsection{\boldmath Proof of Theorem \ref{thm:repgt-main-theorem}, part (b)}\label{sec:triag-pres-building}

To prove part (b) of Theorem \ref{thm:repgt-main-theorem}, fix a triangle presentation $T$ of order $q \geq 2$ and base set $F$. Denote by $\Gamma_T$ the countable group defined by \eqref{eqn:triangle-group} with generating set $S = S_+ \cup S_-$ where $S_+ = \{a_i \mid i \in F\}$ and $S_- = S_+^{-1}$. By \cite[Theorem 3.4]{cmsz1}, the Cayley graph $\Delta_T$ of $\Gamma_T$ with respect to $S$ is the $1$-skeleton of a thick $\tilde{A}_2$-building. By definition, $\Delta_T$ has vertex set $\Gamma_T$ and edge set $\{(g,g h) \mid g \in \Gamma_T, h \in S\}$. When $h \in S_+$, resp.\ $h \in S_-$, we say that $(g,gh)$ is a positive, resp.\ negative edge.

Denote by $\tau : \Gamma_T \recht \ZZ/3\ZZ$ the unique group homomorphism satisfying $\tau(a_i)=1$ for all $i \in F$. Then $\tau$ can be viewed as a type map on the building $\Delta_T$. An automorphism $\al$ of the graph $\Delta_T$ is said to be \emph{type-rotating} if there exist an $r \in \ZZ/3\ZZ$ such that $\tau(\al(v)) = \tau(v) +r$ for all $v \in \Delta_T$. Note that this is equivalent with saying that $\al$ preserves the sign of all edges. Denote by $\AutTr(\Delta_T)$ the totally disconnected locally compact group of all type-rotating automorphisms of $\Delta_T$.

The word problem for $\Gamma_T$ is solved in \cite[Proposition~3.2]{cmsz1}, where it is shown that every element $g \in \Gamma_T$ can be uniquely written as $g=a_{i_1}\cdots a_{i_n} a_{j_1}^{-1}\cdots a_{j_m}^{-1}$ where $n,m \geq 0$ and the indices $i_s,j_r \in F$ satisfy the following two conditions.
\begin{enumerate}[(i)]
\item $i_n\neq j_1$ if $n,m\geq 1$,
\item $i_s\not\to i_{s+1}$ if $1\leq s<n$ and $j_{r+1}\not\to j_r$ if $1 \leq r < m$.
\end{enumerate}

For all $n,m \geq 0$, denote this set of reduced expressions as
\begin{equation}\label{eqn:Snm}
S_{n,m} = \{a_{i_1}\cdots a_{i_n} a_{j_1}^{-1}\cdots a_{j_m}^{-1} \mid i_s,j_r \in F \;\;\text{satisfy conditions (i) and (ii) above}\} \; .
\end{equation}
The following result is also a consequence of \cite[Proposition~3.2]{cmsz1}: the word length of $g \in S_{n,m}$ with respect to $S$ equals $n+m$ and any expression of $g \in S_{n,m}$ as a product of $n+m$ generators must use $n$ generators from $S_+$ and $m$ generators from $S_-$.

So, $S_{n,m}$ can be described geometrically as the set of vertices $g \in \Delta_T$ that are at distance $n+m$ from $e$ and that can be reached by a path involving $n$ positive and $m$ negative edges. In particular, whenever $\al \in \AutTr(\Delta_T)$ and $\al(e) = e$, we get that $\al(S_{n,m}) = S_{n,m}$ for all $n,m \geq 0$. For every vertex $g \in S_{n,m}$, there is a unique path from $e$ to $g$ of length $n+m$ with the first $n$ edges being positive and the last $m$ edges being negative.

Part (b) of Theorem \ref{thm:repgt-main-theorem} is now an immediate consequence of the following result.

\begin{proposition}\label{thm:generic-transitivity}
Let $T$ be a triangle presentation of order $q\geq 2$ with associated building $\Delta_T$ and compact quantum group $\GG_T$ with fundamental representation $u$ of dimension $N = 1+q+q^2$. Put $G=\AutTr(\Delta_T)$ and let $K$ be the stabilizer subgroup of the origin $e\in\Delta_T$. Define $S_{n,m} \subset \Delta_T$ as in \eqref{eqn:Snm}.
\begin{enumerate}[(i)]
\item If the action of $K$ on $S_{1,1}$ is transitive and the action of $K$ on $S_{1,0} = F$ is $2$-transitive, then $\dim(u\bar{u},u\bar{u})=3$, so that both $\Rep(\GG_T)$ and $\widehat{\GG_T}$ have property~(T).
\item Let $n,m \geq 0$. If the action of $K$ on $S_{n,m}$ is transitive, then the vectors $\{e_{i_1 \cdots i_n} \ot e_{j_1 \cdots j_m} \mid a_{i_1}\cdots a_{i_n} a_{j_1}^{-1} \cdots a_{j_m}^{-1} \in S_{n,m}\}$ span an irreducible invariant subspace $\cS_{n,m}$ of $u^{\ot n} \ot {\bar{u}}^{\ot m}$. We have
    \begin{alignat*}{2}
    &\dim(\cS_{n,m}) = |S_{n,m}|= N(N-1) q^{2(n+m-2)} &&\quad\text{if $n,m \geq 1$,}\\
    &\dim(\cS_{n,0}) = |S_{n,0}| = \dim(\cS_{0,n}) = |S_{0,n}| = N q^{2(n-1)} &&\quad\text{if $n \geq 1$.}
    \end{alignat*}
\item If $\Delta_T$ is the Bruhat--Tits $\tilde{A}_2$-building of a commutative local field $\KK$, then the action of $K$ is transitive on each $S_{n,m}$, $n,m \geq 0$ and $2$-transitive on $S_{1,0}$, so that all the conclusions in (i) and (ii) hold.
\end{enumerate}
\end{proposition}

To prove Proposition \ref{thm:generic-transitivity}, we make use of a monoidal functor from $\Rep(\GG_T)$ to the rigid $C^*$-tensor category associated in \cite[\S~2]{arano-vaes} to a locally compact group $G$ and a compact open subgroup $K$. This category $\mathcal{C}_f(K<G)$ of \textit{finite-rank} $L^\infty(G/K)$-$G$-$L^\infty(G/K)$-\textit{modules} consists of Hilbert spaces $\mathcal{H}$ equipped with a unitary representation $\pi$ of $G$ and an $L^\infty(G/K)$-$L^\infty(G/K)$-bimodule structure such that
\[
\pi(g)(f_1\cdot \xi \cdot f_2) = (gf_1)\cdot \pi(g)\xi \cdot (gf_2)\;,
\]
for all $g\in G$, $f_1,f_2\in L^\infty(G/K)$ and $\xi\in\mathcal{H}$, where $G$ acts on $L^\infty(G/K)$ by left translation. The morphisms from $\mathcal{H}$ to $\mathcal{K}$ are the bounded linear maps intertwining the $G$-representations and the $L^\infty(G/K)$-$L^\infty(G/K)$-bimodule structure. We say that $\mathcal{H}$ has \textit{finite rank} whenever $\chi_K\cdot\mathcal{H}$ (or equivalently, $\mathcal{H}\cdot \chi_K$) is finite-dimensional. Here, $\chi_K$ refers to the characteristic function of $K$ considered as an element of $L^\infty(G/K)$.

The tensor product of $\mathcal{H}$ and $\mathcal{K}$ in $\mathcal{C}_f(K<G)$ is the Hilbert space
\begin{align*}
\mathcal{H}\otimes_{L^\infty(G/K)}\mathcal{K} &=\{\xi\in\mathcal{H}\otimes\mathcal{K}\mid \forall gK \in G/K: \xi(\chi_{gK}\otimes  1) = (1\otimes \chi_{gK})\xi\}\\
        &\cong \bigoplus_{g\in G/K} \mathcal{H}\cdot \chi_{gK} \otimes \chi_{gK}\cdot \mathcal{K}\;,
\end{align*}
equipped with the unitary representation $\pi_{\mathcal{H}}\otimes\pi_{\mathcal{K}}$ of $G$, left $L^\infty(G/K)$-module structure induced from the one on $\mathcal{H}$ and right $L^\infty(G/K)$-module structure induced from the one on $\mathcal{K}$. The unit object is $L^2(G/K)$ with the obvious $L^\infty(G/K)$-$G$-$L^\infty(G/K)$-module structure. This turns $\mathcal{C}_f(K<G)$ into a $C^*$-tensor category. It is easy to see that $\mathcal{C}_f(K<G)$ is rigid and that, when $G$ is unimodular, the categorical dimensions are given by $d(\mathcal{H})=\dim(\chi_K\cdot\mathcal{H})$.

Observe that $\Rep(K)$ can be identified with the full $C^*$-tensor subcategory of $\mathcal{C}_f(K<G)$ consisting of those $L^\infty(G/K)$-$G$-$L^\infty(G/K)$-modules with the property that $\chi_{gK} \cdot \xi = \xi \cdot \chi_{gK}$ for all $g \in G$.

As in \cite[\S~2]{arano-vaes}, note that the functor $\cH \mapsto \chi_K \cdot \cH$ identifies $\cC_f(K < G)$ with the category of finite-dimensional $K$-$L^\infty(G/K)$-modules, i.e.\ the category of finite-dimensional covariant representations of $K$ and $L^\infty(G/K)$, which can also be viewed as the category of finite-dimensional $*$-representations of the finite type~I von Neumann algebra $L^\infty(G/K) \rtimes K$.

Consider now the setup of Proposition \ref{thm:generic-transitivity}. Let $G < \AutTr(\Delta_T)$ be a closed subgroup containing $\Gamma_T$ and denote by $K = \Stab_G(e)$ the stabilizer of the vertex $e \in \Delta_T$. For all $n,m \geq 0$, denote by $P_{n,m}$ the set of paths in $\Delta_T$ of length $n+m$ consisting of $n$ positive edges followed by $m$ negative edges. Note that $\AutTr(\Delta_T)$ acts naturally on $P_{n,m}$ and that
\begin{equation}\label{eqn:bijection}
\begin{split}
\Gamma_T \times & F^n \times F^m \recht P_{n,m} : (g,i_1,\ldots,i_n,j_1,\ldots,j_m) \mapsto \\ & (g, ga_{i_1}, g a_{i_1} a_{i_2},\ldots,g a_{i_1} \cdots a_{i_n}, g a_{i_1} \cdots a_{i_n} a_{j_1}^{-1},\cdots, g a_{i_1} \cdots a_{i_n} a_{j_1}^{-1} \cdots a_{j_m}^{-1} )
\end{split}
\end{equation}
is a bijection inducing a unitary $\cU_{n,m} : \ell^2(\Gamma_T) \ot (\CC^F)^{\ot n} \ot (\CC^F)^{\ot m} \recht \ell^2(P_{n,m})$.

Consider the unitary representation of $G$ on $\ell^2(P_{n,m})$ given by the action $G \actson P_{n,m}$. Identifying $\Delta_T = G/K$, we turn $\ell^2(P_{n,m})$ into an $L^\infty(G/K)$-bimodule by using the source and target vertex of a path. In this way, $\ell^2(P_{n,m})$ is an object in $\cC_f(K < G)$.

The following result is a crucial tool to prove Proposition \ref{thm:generic-transitivity}.

\begin{proposition}\label{thm:repgt-functor}
Using the notation introduced above, there is a unique unitary monoidal functor $\bF_T^G$ from $\Rep(\GG_T)$ to $\cC_f(K < G)$ satisfying
\begin{enumerate}[(i)]
\item $\bF_T^G(u^{\ot n}) = \ell^2(P_{n,0})$,
\item for every $W \in (u^{\ot n},u^{\ot n'})$, we have $\bF_T^G(W) = \cU_{n,0} (1 \ot W) \cU_{n',0}^*$.
\end{enumerate}
Moreover, this functor is essentially surjective and satisfies
\begin{enumerate}[(i)]\setcounter{enumi}{2}
\item $\bF_T^G(u^{\ot n} \ot {\bar{u}}^{\ot m}) = \ell^2(P_{n,m})$,
\item for every $W \in (u^{\ot n} \ot {\bar{u}}^{\ot m},u^{\ot n'} \ot {\bar{u}}^{\ot m'})$, we have $\bF_T^G(W) = \cU_{n,m} (1 \ot W) \cU_{n',m'}^*$.
\end{enumerate}
\end{proposition}
\begin{proof}
Write $\cH_n = \ell^2(P_{n,0})$. Note that there is a natural isomorphism of $L^\infty(G/K)$-$G$-$L^\infty(G/K)$-modules $\cH_n \ot_{L^\infty(G/K)} \cH_m = \cH_{n+m}$. To uniquely define the functor $\bF_T^G$, the only nontrivial statement to prove is that $\cU_{n,0} (1 \ot W) \cU_{n',0}^*$ is a morphism in $\cC_f(K < G)$ for all $W \in (u^{\ot n},u^{\ot n'})$. By Remark \ref{rem:Tannaka-point-of-view}, it is sufficient to prove that
$$E_{i,j} := \cU_{i+j+3,0}(1 \ot (1^{\ot i} \ot E \ot 1^{\ot j})) \cU_{i+j,0}^*$$
is a morphism in $\cC_f(K < G)$ for all $i,j \geq 0$. But this operator $E_{i,j}$ from $\cH_{i+j}$ to $\cH_{i+j+3}$ maps the basis vector $e_p$ given by a positive path $p = (g_0,\ldots,g_{i+j})$ of length $i+j$ to the sum of the basis vectors $e_q$ where $q$ runs over all positive paths that arise from $p$ by inserting a positive loop of length $3$ at position $i$. More precisely, the sum runs over all positive paths $q = (h_0,\ldots,h_{i+j+3})$ of length $i+j+3$ satisfying $h_k = g_k$ for all $0 \leq k \leq i$ and $h_k = g_{k-3}$ for all $i+3 \leq k \leq i+j+3$. This geometric description shows that $E_{i,j}$ intertwines the unitary representations of $\AutTr(\Delta_T)$ on $\cH_{i+j}$ and $\cH_{i+j+3}$. Since $E_{i,j}$ also is an $L^\infty(G/K)$-bimodule map, we have uniquely defined $\bF_T^G$.

Using the vectors $\xi_i$ defined in \eqref{eqn:vector-xi}, we defined the isometric intertwiner $Y \in (uu,\bar{u})$ given by $Y e_i = (q+1)^{-1/2} \xi_i$. We similarly have an isometric intertwiner from $\ell^2(P_{0,1})$ to $\ell^2(P_{2,0})$ given by
$$X e_{a_i^{-1}} = (q+1)^{-1/2} \sum_{j,k , (i,j,k) \in T} e_{a_j a_k} \; .$$
We have $\bF_T^G(YY^*) = XX^*$, so that $\bF_T^G$ also satisfies the extra properties in (iii) and (iv).

Since $\ell^2(P_{n,0})$ belongs to the image of $\bF_T^G$, to prove the essential surjectivity of $\bF_T^G$, it suffices to prove that for every $g \in G$, every irreducible representation of $K_g := K \cap g K g^{-1}$ appears as a subrepresentation of $K_g \actson \chi_K \cdot \ell^2(P_{n,0}) \cdot \chi_{gK}$ for some $n \in \NN$. So, it suffices to prove that the action of $K_g$ on the set of positive paths from $e$ to $g \in \Delta_T$ is faithful. Since for any $h \in \Delta_T$, there exists a positive path from $e$ to $g$ that passes through $h$, this result follows because by definition, the action of $K_g$ on $\Delta_T$ is faithful.
\end{proof}

We are now ready to prove Proposition \ref{thm:generic-transitivity}.

\begin{proof}[Proof of Proposition \ref{thm:generic-transitivity}]
Throughout the proof, we make use of the basic intertwiners introduced in Lemma \ref{thm:basic-intertwiners}. In particular, we denote by $R \in (uu,uu)$ the orthogonal projection of $\CC^N \ot \CC^N$ onto the subspace spanned by $e_{ij}, i \to j$. We also define $\bar{R} \in (\bar{u} \bar{u},\bar{u}\bar{u})$ as the projection of $\CC^N \ot \CC^N$ onto the subspace spanned by $e_{ij}, j \to i$. Slightly modifying the notation, we denote by $Q \in (u\bar{u},u\bar{u})$ the projection of $\CC^N \ot \CC^N$ onto the subspace spanned by $e_i \ot e_i$. As in the formulation of the proposition, we write $G = \AutTr(\Delta_T)$ and define $K < G$ as the stabilizer of the origin $e \in \Delta_T$. Denote by $\bF_T^G$ the monoidal functor introduced in Proposition \ref{thm:repgt-functor}.

We start by proving (ii). With some abuse of notation, denote by $S_{n,m}$ the orthogonal projection of $(\CC^N)^{\ot n} \ot (\CC^N)^{\ot m}$ onto the linear span $\cS_{n,m}$ of the vectors
$$\{e_{i_1 \cdots i_n} \ot e_{j_1 \cdots j_m} \mid a_{i_1}\cdots a_{i_n} a_{j_1}^{-1} \cdots a_{j_m}^{-1} \in S_{n,m}\} \; .$$
For $n,m \geq 2$, we have that
\begin{align*}
S_{n,m} = 1 - & \bigvee_{0 \leq a \leq n-2} (1^{\ot a} \ot R \ot 1^{\ot (n+m-a-2)}) \;\; \vee \;\; (1^{\ot(n-1)} \ot Q \ot 1^{\ot (m-1)}) \;\; \vee \\
& \bigvee_{0 \leq b \leq m-2} (1^{\ot (n+b)} \ot \bar{R} \ot 1^{\ot (m-b-2)})\; .
\end{align*}
So, $S_{n,m} \in (u^{\ot n} \ot {\bar{u}}^{\ot m} , u^{\ot n}\ot {\bar{u}}^{\ot m})$ and $\cS_{n,m}$ is an invariant subspace of $u^{\ot n} \ot {\bar{u}}^{\ot m}$. Assume now that the action of $K$ on $S_{n,m}$ is transitive. We prove that $\cS_{n,m}$ is irreducible. It suffices to prove that $\bF_T^G(\cS_{n,m})$ is irreducible in $\cC_f(K < G)$. This means that we have to prove that $\chi_{K} \cdot \bF_T^G(\cS_{n,m})$
is an irreducible $K$-$L^\infty(G/K)$-module. We identify $\chi_{K} \cdot \bF_T^G(\cS_{n,m})$ with the $\ell^2$-space of the set of geodesics in $\Delta_T$ of length $n+m$ starting with $n$ positive edges, followed by $m$ negative edges. As explained above, these geodesics are entirely determined by their end point $g \in S_{n,m}$. So, we have identified $\chi_{K} \cdot \bF_T^G(\cS_{n,m})$ with the $K$-$L^\infty(G/K)$-module $\ell^2(S_{n,m})$. An endomorphism of this $K$-$L^\infty(G/K)$-module is given by a $K$-invariant function in $\ell^\infty(S_{n,m})$. By the transitivity of $K \actson S_{n,m}$, the irreducibility follows.

We now prove (i). By part (a) of Theorem \ref{thm:repgt-main-theorem}, it suffices to prove that $\dim (u \bar{u},u \bar{u}) = 3$. Using the notation of part (ii) proved above, we write $\CC^N \ot \CC^N = \cL \oplus \cS_{1,1}$, where $\cL$ is the linear span of $\{e_i \ot e_i \mid i \in F\}$. By part (ii), the invariant subspace $\cS_{1,1}$ is an irreducible representation of dimension $N (N-1)$. By definition, $\cL$ contains the trivial representation. It suffices to prove that its orthogonal complement in $\cL$, which has dimension $N-1$, is irreducible. So it suffices to prove that $\bF_T^G(\cL)$ splits in $\cC_f(K < G)$ as the sum of the trivial module and an irreducible module of dimension $N-1$. Note that $\chi_K \cdot \bF_T^G(\cL)$ can be identified with the $\ell^2$-space of the paths of length $2$ of the form $(e,a,e)$ with $a \in F$. To conclude the proof of (i), we have to prove that the unitary representation of $K$ on $\ell^2(F)$ is the sum of the trivial representation and an irreducible representation. Since by hypothesis, the action $K \actson F = S_{1,0}$ is $2$-transitive, the conclusion follows.

Proof of (iii). Assume that $\Delta_T$ is the Bruhat--Tits $\tilde{A}_2$-building of a commutative local field $\KK$ with ring of integers $\cO$ and residue field $\FF_q$. By Remark~1 in \cite[p.~242]{cms}, the action of $K$ on $S_{n,m}$ is transitive for all $n,m \geq 0$. Also, $\PGL(3,\cO) \hookrightarrow K$ and $F$ can identified with the projective plane $\PG(2,\FF_q)$ of $\FF_q$ in such a way that the action of $\PGL(3,\cO)$ on $F$ is given by composing the quotient homomorphism $\PGL(3,\cO) \recht \PGL(3,\FF_q)$ with the natural action of $\PGL(3,\FF_q)$ on $\PG(2,\FF_q)$, which is $2$-transitive. This concludes the proof of the proposition.
\end{proof}

\section{\boldmath Classification, categories \texorpdfstring{$\mathcal{C}_f(K<G)$}{C\_f(K<G)} and bicrossed products}\label{sec:classification}

To every triangle presentation $T$ of order $q \geq 2$, we associated in Definition \ref{def:G-T} the compact quantum group $\GG_T$. As we recalled in Section \ref{sec.triangle-pres} from \cite{cmsz1}, we denote by $\Gamma_T$ the discrete group whose presentation is given by $T$ and by $\Delta_T$ its Cayley graph, which is an $\tilde{A}_2$-building. By construction, $\Gamma_T$ acts simply transitively on $\Delta_T$. As in Section \ref{sec:triag-pres-building}, we denote by $\AutTr(\Delta_T)$ the group of type-rotating automorphisms of $\Delta_T$. For every closed subgroup $G < \AutTr(\Delta_T)$ containing $\Gamma_T$ and denoting by $K = \Stab_G(e)$ the stabilizer of the vertex $e \in \Delta_T$, we constructed in Proposition \ref{thm:repgt-functor} the monoidal functor $\bF : \Rep(\GG_T) \recht \cC_f(K < G)$. This functor played an essential role in the proof of part (b) of Theorem \ref{thm:repgt-main-theorem}.

Since the action of $\Gamma_T$ on $\Delta_T$ is simply transitive, we have $G = \Gamma_T K$ and $\Gamma_T \cap K = \{e\}$. So, $G = \Gamma_T K$ is a \emph{matched pair} of the discrete group $\Gamma_T$ and the compact group $K$, in the sense of \cite{Kac,baaj-skandalis,vaes-vainerman,fima-bicrossed}. The \emph{bicrossed product construction} associates to the matched pair $G = \Gamma_T K$ a compact quantum group $\HH$ with underlying von Neumann algebra $L^\infty(\HH) = L^\infty(G/\Gamma_T) \rtimes \Gamma_T$ and dual von Neumann algebra $\ell^\infty(\hat{\HH}) = \ell^\infty(G/K) \rtimes K$. In Proposition \ref{thm:kac-type-ff}, we prove that $\cC_f(K < G)$ precisely is the representation category $\Rep(\HH)$.

This means that we have natural quotient morphisms of discrete quantum groups
\begin{equation}\label{eqn:quotients}
\widehat{\GG_T} \recht \widehat{\HH} \recht \Gamma_T \; .
\end{equation}

When $T$ satisfies the hypotheses of Theorem \ref{thm:repgt-main-theorem}, the discrete quantum group $\widehat{\GG_T}$ has property~(T). Hence also its quotient $\widehat{\HH}$ has property~(T). This follows from \cite[Proposition 9.6]{pv-repr-subfactors} and because property~(T) of $\Rep(\HH)$ is equivalent with property~(T) for $\widehat{\HH}$ in the Kac case. So triangle presentations also give rise to numerous bicrossed products with property~(T). They are quite different from the examples in \cite{fima-bicrossed}, which are given by matched pairs of a countable group $\Gamma$ and a finite group $K$ and hence give discrete quantum groups that are commensurable with the group $\Gamma$.

More generally, for any locally compact group $G$ with compact open subgroup $K < G$, we associate in Proposition \ref{thm:kac-type-ff} a fiber functor $\bF : \cC_f(K < G) \recht \Hilb_f$ to any subgroup $\Gamma < G$ satisfying $G = \Gamma K$ and $\Gamma \cap K = \{e\}$. We prove that the compact quantum group associated to this fiber functor by Woronowicz' Tannaka--Krein theorem is the bicrossed product of the matched pair $G = \Gamma K$.

This brings up the interesting open question whether the converse holds (see Remark \ref{rem:question-fiber-functor}): does the existence of a fiber functor on $\cC_f(K < G)$ imply the existence of a subgroup $\Gamma < G$ satisfying $G = \Gamma K$ and $\Gamma \cap K = \{e\}$~?

We then return to triangle presentations $T$ of order $q \geq 2$ for which the associated $\tilde{A}_2$-building is Bruhat--Tits. Associated to this data, we have the compact quantum group $\GG_T$ and the compact quantum group $\HH_T$ given by the matched pair $G_T = \Gamma_T K_T$ where $G_T = \AutTr(\Delta_T)$ and $K_T = \Stab_{G_T}(e)$. In Theorem \ref{thm:classif-compact}, we classify these compact quantum groups $\GG_T$ and $\HH_T$ up to isomorphism and also, partially, up to monoidal equivalence.

Using the classification of triangle presentations of order $q=2$ or $3$ in \cite{cmsz1,cmsz2}, Theorem \ref{thm:classif-compact} provides numerous new examples of nonisomorphic, monoidally equivalent compact quantum groups. All their duals are discrete quantum groups with property~(T).

It follows in particular from Theorem \ref{thm:classif-compact} that for the triangle presentations $T_1$ and $T_2$ in Table~\ref{tbl:triang-pres-examples}, the tensor categories $\Rep(\HH_{T_i}) = \cC_f(K_{T_i} < G_{T_i})$ are not monoidally equivalent. As we explain in Remark \ref{rem:indication-mon-equiv} below, we conjecture that the representation categories $\Rep(\GG_{T_i})$ are monoidally equivalent and we provide numerical evidence for this conjecture. Altogether, it is thus very unlikely that the functors $\bF_i : \Rep(\GG_{T_i}) \recht \cC_f(K_{T_i} < G_{T_i})$ are equivalences of categories. We were however unable to prove this. Given a triangle presentation $T$ of order $q \geq 2$, it is even unclear whether the linear span of $\{e_{i_1 \cdots i_n} \mid a_{i_1} \cdots a_{i_{3n}} = e\}$ is an invariant subspace of the representation $u^{\ot 3n}$ of $\GG_T$. Its image under $\bF$ would be the natural $L^\infty(G/K)$-$G$-$L^\infty(G/K)$-module defined as the $\ell^2$-space of all \emph{closed} positive paths of length $3n$ in $\Delta_T$.

Let $G$ be a locally compact group with compact open subgroup $K < G$. The tensor category $\mathcal{C}_f(K<G)$ comes with the natural faithful unitary functor to the category of finite-dimensional Hilbert spaces mapping a finite-rank $L^\infty(G/K)$-$G$-$L^\infty(G/K)$-module $\mathcal{H}$ to the finite-dimensional Hilbert space $\chi_K\cdot\mathcal{H}$. In general, there is no canonical way to turn this functor into a tensor functor, but in the presence of a subgroup $\Gamma < G$ satisfying $G = \Gamma K$ and $\Gamma \cap K = \{e\}$, this can be done as follows.

\begin{proposition}\label{thm:kac-type-ff}
Let $G$ be a locally compact group, and $K<G$ a compact open subgroup.
Let $\mathbf{F}$ be the functor from $\mathcal{C}_f(K<G)$ to $\mathrm{Hilb}_f$ that sends $\mathcal{H}$ to $\chi_K\cdot\mathcal{H}$ and acts on morphisms by restriction.
Given a subgroup $\Gamma< G$ such that $G=\Gamma K$ with $K\cap \Gamma=\{e\}$, the natural unitary
\begin{equation}\label{eqn:kac-type-ff-unitary}
        \chi_K\cdot (\mathcal{H}\otimes_{L^\infty(G/K)} \mathcal{K})\to (\chi_K\cdot \mathcal{H}) \otimes (\chi_K\cdot \mathcal{K}): \xi\mapsto \sum_{\gamma\in \Gamma} (1\otimes \pi_{\mathcal{K}}(\gamma)^*\chi_{\gamma K})\xi
\end{equation}
turns $\mathbf{F}$ into a dimension-preserving fiber functor. The Kac type compact quantum group associated to this fiber functor is the bicrossed product of the matched pair $G = \Gamma K$.
\end{proposition}
\begin{proof}
Since $\Gamma$ is a cocompact lattice in $G$, the group $G$ is unimodular. The condition on the position of $\Gamma$ in $G$ implies that $\Gamma\to G/K:\gamma\mapsto \gamma K$ is a bijection.
Verifying that \eqref{eqn:kac-type-ff-unitary} gives $\mathbf{F}$ the structure of a monoidal functor then boils down to the identity
\begin{align*}
\sum_{\gamma,\gamma'\in \Gamma} &(1\otimes \pi_{\mathcal{K}}(\gamma)^*\otimes \pi_{\mathcal{L}}(\gamma\gamma')^*)(1\otimes\chi_{\gamma K}\otimes \chi_{\gamma\gamma' K})\xi \numberthis\label{eqn:kac-type-ff-unitary-verif}\\
&= \sum_{\gamma,\gamma'\in \Gamma} (1\otimes \pi_{\mathcal{K}}(\gamma)^*\otimes \pi_{\mathcal{L}}(\gamma')^*)(1\otimes\chi_{\gamma K}\otimes \chi_{\gamma' K})\xi\;.
\end{align*}
So, $\bF$ is a fiber functor, which is dimension preserving because $G$ is unimodular.

As in \cite[Proposition 2.2]{arano-vaes}, the map $\cH \recht \chi_K \cdot \cH$ identifies $\mathcal{C}_f(K<G)$ with the category of finite-dimensional $K$-$L^\infty(G/K)$-modules, i.e.\ the category $\Rep_f(M)$ of finite-dimensional representations of the finite type I von Neumann algebra $M = L^\infty(G/K) \rtimes K$. The monoidal structure in \eqref{eqn:kac-type-ff-unitary} then translates into the co-associative $*$-homomorphism $\Delta : M \recht M \ovt M$ given by
\begin{equation}\label{eqn:comult-bicrossed}
\begin{split}
&\Delta(\chi_{\gamma K}) = \sum_{\gamma' \in \Gamma} \chi_{\gamma' K} \ot \chi_{{\gamma'}^{-1} \gamma K} \; ,\\
&\Delta(u_k) = \sum_{\gamma \in \Gamma} u_k \chi_{\gamma K} \ot u_{\om(k,\gamma)} \; ,
\end{split}
\end{equation}
where for each $k \in K$ and $\gamma \in \Gamma$, we denote by $\om(k,\gamma)$ the unique element in $K$ such that $k \gamma \in \Gamma \om(k,\gamma)$.

Denoting by $\HH$ the compact quantum group arising as the bicrossed product of the matched pair $G = \Gamma K$, we see that its dual discrete quantum group is identical to $(M,\Delta)$. This concludes the proof of the proposition.
\end{proof}

\begin{remark}\label{rem:question-fiber-functor}
Let $G$ be a locally compact group with compact open subgroup $K < G$. Does the converse of Proposition \ref{thm:kac-type-ff} hold? More precisely, does the existence of a fiber functor on $\cC_f(K < G)$ force the existence of a subgroup $\Gamma < G$ satisfying $G = \Gamma K$ and $\Gamma \cap K = \{e\}$. Note that the existence of such a subgroup $\Gamma$ forces $G$ to be unimodular. In particular, we do not know whether for non-unimodular $G$, the tensor category $\cC_f(K < G)$ can ever admit a fiber functor.

When $G$ is unimodular and $K < G$ is a compact open subgroup, the question of the existence of a \emph{dimension preserving} fiber functor on $\cC_f(K < G)$ can be rephrased as the following $3$-cohomology problem. Define the finite type I von Neumann algebra $M = L^\infty(G/K) \rtimes K$. As in the proof of Proposition \ref{thm:kac-type-ff}, we can view $\cC_f(K < G)$ as the category $\Rep_f(M)$ of finite-dimensional representations of $M$. Choosing a section $\theta : G/K \recht G$, we define the $1$-cocycle $\om : G \times G/K \recht K$ such that
$$g \, \theta(hK) = \theta(ghK) \, \om(g,hK) \quad\text{for all}\;\; g \in G, hK \in G/K \; .$$
We can then define the normal unital $*$-homomorphism $\Delta : M \recht M \ovt M$ given by
\begin{equation*}
\begin{split}
&\Delta(\chi_{gK}) = \sum_{hK \in G/K} \chi_{hK} \ot \chi_{\theta(hK)^{-1} gK} \; ,\\
&\Delta(u_k) = \sum_{gK \in G/K} u_k \chi_{gK} \ot u_{\om(k,gK)} \; .
\end{split}
\end{equation*}
By construction, under the identification of $\cC_f(K < G)$ with $\Rep_f(M)$, the tensor product of $\pi_1,\pi_2 \in \Rep_f(M)$ becomes $(\pi_1 \ot \pi_2) \circ \Delta$. Although the tensor product in $\cC_f(K < G)$ is naturally associative, this does not imply that $\Delta$ is co-associative. We rather find that
$$(\Delta \ot \id) \circ \Delta = (\Ad A) \circ (\id \ot \Delta) \circ \Delta$$
where $A \in \cU(M \ovt M \ovt M)$ is the unitary \emph{co-associator} given by
$$A = \sum_{gK,hK \in G/K} \chi_{gK} \ot \chi_{hK} \ot u_{\om(\theta(gK),hK)} \; .$$
The unitary $A$ satisfies a $3$-cocycle relation. The question whether there exists a dimension preserving fiber functor on $\cC_f(K < G)$ is now equivalent with the question whether $A$ is a coboundary, i.e.\ whether there exists a unitary $\Om \in \cU(M \ovt M)$ satisfying
\begin{equation}\label{eqn:coboundary}
A = (\Delta \ot \id)(\Om)^* \, (\Om^* \ot 1) \, (1 \ot \Om) \, (\id \ot \Delta)(\Om) \; .
\end{equation}
The existence of a subgroup $\Gamma < G$ with $\Gamma \cap K = \{e\}$ and $\Gamma K = G$ is equivalent with the existence of a map $\rho : G/K \recht K$ such that the new section $\theta' : G/K \recht G$ given by $\theta'(gK) = \theta(gK) \, \rho(gK)$ has the property that the image of $\theta'$ is a subgroup of $G$. In that case, the unitary
$$\Om = \sum_{gK \in G/K} \chi_{gK} \ot u_{\rho(gK)}^*$$
satisfies \eqref{eqn:coboundary}.
\end{remark}

\begin{remark}
As explained after \eqref{eqn:quotients}, triangle presentations give rise to numerous bicrossed product quantum groups with property~(T). In the Bruhat--Tits case, all this can be made more concrete. So let $\KK$ be a commutative local field with ring of integers $\cO$. Put $G = \PGL(3,\KK)$ and $K = \PGL(3,\cO)$. Let $\Gamma < G$ be a subgroup such that $G = \Gamma K$ and $\Gamma \cap K = \{e\}$. The compact bicrossed product $\HH$ associated with the matched pair $G = \Gamma K$ has underlying von Neumann algebra $L^\infty(G/\Gamma) \rtimes \Gamma$, which has property~(T) by \cite[Theorem 1.3]{bouljihad} and because $\Gamma$ is a property~(T) group. The representation category $\Rep(\HH)$ is equivalent with $\cC_f(K < G)$, which has property~(T) by \cite[Proposition 4.2]{arano-vaes}. Finally, by \cite[Theorem 3.1]{cmsz1} and because $\Gamma$ can be viewed as acting simply transitively on the building $G/K$, there is a canonically defined triangle presentation $T$ such that $\Gamma = \Gamma_T$. Then, $\widehat{\GG_T}$ has property~(T) by Theorem \ref{thm:repgt-main-theorem} and so does its quotient $\widehat{\HH}$. Altogether, this provides three quite different proofs of property~(T) for $\widehat{\HH}$.
\end{remark}

We now state and prove a classification theorem for the compact quantum groups associated with a triangle presentation. For the formulation of this result, note that for every triangle presentation $T$ on a base set $F$, the \emph{opposite} $\{(a,b,c) \in F^3 \mid (c,b,a) \in T\}$ also is a triangle presentation. We say that two triangle presentations $T$ on $F$ and $T'$ on $F'$ are isomorphic if there exists a bijection $\al : F \recht F'$ satisfying $(\al \times \al \times \al)(T) = T'$.

\begin{theorem}\label{thm:classif-compact}
For $i=1,2$, let $T_i$ be triangle presentations of order $q_i \geq 2$ with associated compact quantum group $\GG_i := \GG_{T_i}$ and triangle group $\Gamma_i$. Assume that the buildings $\Delta_{T_i}$ are the Bruhat--Tits $\tilde{A}_2$-building of a commutative local field $\KK_i$. Denote by $\HH_i$ the compact bicrossed product given by the matched pair $G_i = \Gamma_i K_i$, where $G_i = \AutTr(\Delta_{T_i})$ and $K_i < G_i$ is the stabilizer of the origin $e \in \Delta_{T_i}$.
\begin{enumerate}[(i)]
\item The compact quantum groups $\HH_1$ and $\HH_2$ are monoidally equivalent if and only if the local fields $\KK_1$ and $\KK_2$ are isomorphic.
\item If the compact quantum groups $\GG_1$ and $\GG_2$ are monoidally equivalent, then $q_1 = q_2$. If the local fields $\KK_1$ and $\KK_2$ are isomorphic, then $\GG_1$ and $\GG_2$ are monoidally equivalent.
\item The following statements are equivalent.
\begin{enumerate}[a)]
\item The compact quantum groups $\GG_1$ and $\GG_2$ are isomorphic.
\item The compact quantum groups $\HH_1$ and $\HH_2$ are isomorphic.
\item We have $q_1 = q_2$ and the triangle presentation $T_1$ is isomorphic to $T_2$ or its opposite.
\end{enumerate}
\end{enumerate}
\end{theorem}

Note that (ii) only provides a partial result. As we explain in Remark \ref{rem:indication-mon-equiv}, we have convincing evidence, but no proof, that $\GG_1$ and $\GG_2$ are monoidally equivalent if and only if $q_1 = q_2$.

\begin{proof}
Denote by $\cO_i \subset \KK_i$ the ring of integers. Combining \cite[Theorem 1.26]{descent-in-buildings} with the fact that the spherical building at infinity for $\Delta_{T_i}$ is the projective plane of $\KK_i$, it follows that $G_i = \PGL(3,\KK_i) \rtimes \Aut(\KK_i)$ and $K_i = \PGL(3,\cO_i) \rtimes \Aut(\KK_i)$. Note here that the field automorphisms $\sigma \in \Aut(\KK_i)$ automatically preserve $\cO_i$ globally and that they form a compact group.

Denote by $u_i$ the fundamental representation of the compact quantum group $\GG_i$. Denote by $\bF_i : \Rep(\GG_i) \recht \cC_f(K_i < G_i)$ the monoidal functor given by Proposition \ref{thm:repgt-functor}. Write $N_i = 1+q_i+q_i^2$.

(i) When the local fields $\KK_i$ are isomorphic, the buildings $\Delta_{T_i}$ are isomorphic and therefore, the inclusions $K_i < G_i$ are isomorphic. It follows that the rigid $C^*$-tensor categories $\cC_f(K_i < G_i)$ are monoidally equivalent, so that also the compact quantum groups $\HH_1$ and $\HH_2$ are monoidally equivalent.

Conversely, assume that the tensor categories $\cC_f(K_i < G_i)$ are monoidally equivalent. By Lemmas \ref{lem:characterize-Rep-K} and \ref{lem:elementary} below, this monoidal equivalence restricts to a monoidal equivalence of $\Rep(K_1)$ and $\Rep(K_2)$. By Lemma \ref{lem:play-with-dim}, the smallest possible dimension of an object in $\cC_f(K_i < G_i) \setminus \Rep(K_i)$ is $N_i$. So, $N_1 = N_2$ and thus, $q_1 = q_2$.
By \cite[Theorem 3]{klingenberg}, the groups $\PGL(3,\cO_i)$ have no nontrivial abelian normal subgroups. Then also the groups $K_i$ have no nontrivial abelian normal subgroups. Since $\Rep(K_1)$ and $\Rep(K_2)$ are monoidally equivalent, it follows from \cite[Theorem 3.1]{nesh-tus-cohom} that $K_1 \cong K_2$.

When $\KK_i$ has characteristic zero, it is isomorphic with a finite extension of $\QQ_p$ for some prime $p$. When $\KK_i$ has positive characteristic, it is isomorphic with $\FF_{q_i}((X))$. In the first case, $\Aut(\KK_i)$ is a finite group and the image of $\SL(3,\cO_i)$ in $\PGL(3,\cO_i)$ is an open (and thus, finite index) subgroup, so that $K_i$ and $\SL(3,\cO_i)$ admit isomorphic open subgroups. In particular, every open subgroup of $K_i$ has a finite abelianization. In the second case, $\Aut(\KK_i) \cong \cO_i^\times \rtimes \Aut(\FF_{q_i})$ so that $K_i$ has an open subgroup that admits $\cO_i^\times$ as a quotient. Since $K_1 \cong K_2$, we conclude that either both $\KK_i$ have characteristic zero, or both have positive characteristic. In the second case, because $q_1 = q_2$, it follows that $\KK_1 \cong \KK_2$. In the first case, it follows from the discussion above that $\SL(3,\cO_1)$ and $\SL(3,\cO_2)$ admit isomorphic open subgroups. It then follows from \cite[Corollary 0.3]{pink} that $\KK_1 \cong \KK_2$.

(ii) First assume that $\al : \Rep(\GG_1) \recht \Rep(\GG_2)$ is an equivalence of $C^*$-tensor categories. By symmetry, we may assume that $q_1 \leq q_2$. Then, $\al(u_1)$ is an irreducible representation of $\GG_2$ that generates $\Rep(\GG_2)$. Therefore, the object $\bF_2(\al(u_1))$ generates $\cC_f(K_2 < G_2)$. This implies that $\bF_2(\al(u_1)) \not\in \Rep(K_2)$. By Lemma \ref{lem:play-with-dim} below, $\dim(\bF_2(\al(u_1))) \geq N_2$ so that
$$N_1 = \dim(\al(u_1)) = \dim(\bF_2(\al(u_1))) \geq N_2 \; .$$
So $q_1 \geq q_2$ and we conclude that $q_1 = q_2$.

Secondly, assume that $\KK_1 \cong \KK_2$. In particular, $q_1 = q_2$ and we write $N = N_1 = N_2$. Since $\KK_1 \cong \KK_2$, we find an isomorphism of graphs $\theta : \Delta_{T_1} \recht \Delta_{T_2}$. Since the action $\Gamma_2 \actson \Delta_{T_2}$ is vertex transitive, we may assume that $\theta(e) = e$. Denote by $\tau_i : \Gamma_i \recht \ZZ/3\ZZ$ the unique group homomorphism sending the canonical generators to $1$. Then, $\tau_i$ can be viewed as a type map on $\Delta_{T_i}$, in the sense that for every triangle in $\Delta_{T_i}$, the types of the three vertices of the triangle are $\{0,1,2\}$. Since the type map on an $\tilde{A}_2$-building is uniquely determined by the types of the vertices of a single triangle and since $\theta(e) = e$, we conclude that either $\theta$ is type preserving or $\theta$ preserves type $0$ and interchanges types $1$ and $2$. In the second case, we replace $T_2$ by its opposite, which does not change the compact quantum group $\GG_2$ up to isomorphism. So, we may assume that there is a type preserving isomorphism of graphs $\theta : \Delta_{T_1} \recht \Delta_{T_2}$ satisfying $\theta(e) = e$. Since $\theta$ is type preserving, $\theta$ maps positive edges to positive edges.

Write $F = \{1,\ldots,N\}$. For $i \in \{1,2\}$, denote by $L_{i,n}$ the set of paths in $\Delta_{T_i}$ that start at $e$ and consist of $n$ positive edges. Note that
$$F^n \recht L_{i,n} : (a_{i_1},\ldots,a_{i_n}) \mapsto (e,a_{i_1}, a_{i_1} a_{i_2}, \ldots, a_{i_1} \cdots a_{i_n})$$
is a bijection and that $\theta$ induces a bijection $L_{1,n} \recht L_{2,n}$. Altogether, this induces a permutation of $F^n$ and thus a family of unitary operators $\cU_n : (\CC^N)^{\ot n} \recht (\CC^N)^{\ot n}$. Defining the vectors $E_i$ by \eqref{eqn:triangle-invariant-vector} and reasoning as in the proof of Proposition \ref{thm:repgt-functor}, we find that
$$\cU_{a+b+3} \; (1^{\ot a} \ot E_1 \ot 1^{\ot b}) \; \cU_{a+b}^* = 1^{\ot a} \ot E_2 \ot 1^{\ot b}$$
and we conclude that $\Rep(\GG_1)$ is monoidally equivalent with $\Rep(\GG_2)$.

(iii) We first prove the easy implications c) $\Rightarrow$ a) and c) $\Rightarrow$ b). When two triangle presentations are isomorphic, all data are isomorphic and we find isomorphisms $\GG_1 \cong \GG_2$ and $\HH_1 \cong \HH_2$. When $T_2$ equals the opposite of $T_1$, there is a natural isomorphism $\GG_1 \recht \GG_2$ given by sending the fundamental representation of $\GG_1$ to the contragredient of the fundamental representation of $\GG_2$. There then also is a natural isomorphism $\Gamma_1 \recht \Gamma_2$ given by sending the generators of $\Gamma_1$ to the inverses of the generators of $\Gamma_2$. This isomorphism induces an isomorphism between the Cayley graphs $\Delta_{T_1} \recht \Delta_{T_2}$, so that we find a continuous group isomorphism $\theta : G_1 \recht G_2$ satisfying $\theta(\Gamma_1) = \Gamma_2$ and $\theta(K_1) = K_2$. This isomorphism $\theta$ induces an isomorphism $\HH_1 \cong \HH_2$.

We next prove that a) $\Rightarrow$ c) and b) $\Rightarrow$ c). Assume that a) or b) holds. From (i) and (ii), we know that $q_1 = q_2$ and we write $N = N_1 = N_2$. Define as follows the unitary representation $u \in M_N(\CC) \ot C(\HH_2)$. When b) holds, we define $u$ as the image of the fundamental representation of $\HH_1$ under the isomorphism $\HH_1 \cong \HH_2$. When a) holds, we define $u$ by first applying the isomorphism $\GG_1 \cong \GG_2$ to the fundamental representation of $\GG_1$ and then applying the functor $\bF_2 : \Rep(\GG_2) \recht \cC_f(K_2 < G_2)$ and identifying $\cC_f(K_2 < G_2)$ with $\Rep(\HH_2)$ through Proposition~\ref{thm:kac-type-ff}.

In both cases, $u$ is an $N$-dimensional representation of $\HH_2$ that generates $\Rep(\HH_2)$ and the vector $E_1 = \sum_{(i,j,k) \in T_1} e_i \ot e_j \ot e_k$ is an invariant vector for the $3$-fold tensor power $u^{\ot 3}$. In particular, $u \not\in \Rep(K_2)$. Combining Lemma \ref{lem:play-with-dim} and Proposition \ref{thm:kac-type-ff}, we find an orthonormal basis $(\eta_i)_{i=1,\ldots,N}$ of $\CC^N$ and a map $\om : T_2 \recht S^1$ such that
\begin{equation}\label{eq:dichotomy-E1}
E_1 = \sum_{(i,j,k) \in T_2} \om(i,j,k) \, \eta_i \ot \eta_j \ot \eta_k \qquad\text{or}\qquad E_1 = \sum_{(i,j,k) \in T_2} \om(i,j,k) \, \eta_k \ot \eta_j \ot \eta_i \; .
\end{equation}
Assume that the first equality holds. Define a unitary $w \in \cU(\CC^N)$ such that $w(\eta_i) = e_i$ for all $i$. Using the notation in \eqref{eq:vector-E-om}, we find that $(w \ot w \ot w)E_1 = E_{2,\om}$. By Lemma \ref{lem:rigid-intertwiners}, it follows that $T_1$ and $T_2$ are isomorphic. If the second equality in \eqref{eq:dichotomy-E1} holds, we similarly find that $T_1$ is isomorphic with the opposite of $T_2$.
\end{proof}

In the proof of Theorem \ref{thm:classif-compact}, we made use of a few lemmas that we prove now. We first introduce the following definition.

\begin{definition}\label{def:support-module}
Let $K$ be a compact open subgroup of a locally compact group $G$. For every $\cH \in \cC_f(K < G)$ we define $\supp(\cH) \subset G$ as the set of all $g \in G$ such that $\chi_K \cdot \cH \cdot \chi_{gK} \neq \{0\}$.
\end{definition}

Note that $\supp(\cH)$ is both left $K$-invariant and right $K$-invariant. Also, $\supp(\cH)$ consists of finitely many right (or left) $K$-cosets. It is easy to check that $\supp(\cH \ot_{L^\infty(G/K)} \cK) = \supp(\cH) \, \supp(\cK)$ and that $\supp(\cH \oplus \cK) = \supp(\cH) \cup \supp(\cK)$.

\begin{lemma}\label{lem:characterize-Rep-K}
Let $G$ be a totally disconnected locally compact group with compact open subgroup $K < G$. Assume that $\bigcup_{n=1}^\infty (KgK)^n$ is noncompact for every $g \in G \setminus K$. Then $\Rep(K) \subset \cC_f(K < G)$ has the following intrinsic characterization: every full $C^*$-tensor subcategory of $\cC_f(K < G)$ with finitely many irreducible objects is contained in $\Rep(K)$ and $\Rep(K)$ is the union of an increasing family of such subcategories.
\end{lemma}
\begin{proof}
Since $K$ is totally disconnected, we can write $\Rep(K)$ as an increasing union of full $C^*$-tensor subcategories of the form $\Rep(K_i)$ where the $K_i$ are finite quotients of $K$. The hypothesis of the lemma implies that every module $\cH$ that is contained in a full $C^*$-tensor subcategory with finitely many irreducibles must satisfy $\supp(\cH) = K$. This precisely means that $\cH \in \Rep(K)$.
\end{proof}

\begin{lemma}\label{lem:elementary}
Let $\KK$ be a commutative local field with ring of integers $\cO$. Denote by $\Delta$ the associated Bruhat--Tits $\tilde{A}_2$-building. Let $G < \AutTr(\Delta)$ be a closed subgroup containing $\PGL(3,\KK)$. Denote by $K = \Stab_G(e)$ the stabilizer of the origin. For every $g \in G \setminus K$, we have that $\bigcup_{n = 1}^\infty (KgK)^n$ is noncompact.
\end{lemma}
\begin{proof}
Fix a generator $\varpi$ of the maximal ideal of $\cO$. Writing $K_0 = \PGL(3,\cO)$, every element $g \in \PGL(3,\KK)$ is of the form
\begin{equation}\label{eqn:writing-g}
g = k_0 \, D(a,b) \; k_1 \quad\text{with}\quad D(a,b) = \begin{pmatrix} \varpi^{-a} & 0 & 0 \\ 0 & \varpi^{-b} & 0 \\ 0 & 0 & 1 \end{pmatrix} \quad\text{and $k_0,k_1 \in K_0$, $a,b \geq 0$.}
\end{equation}
Since the action of $\PGL(3,\KK)$ on $\Delta$ is transitive, every $g \in G$ can be written as in \eqref{eqn:writing-g} with $k_0,k_1 \in K$ and $a,b \geq 0$. So when $g \in G \setminus K$, the double coset $K g K$ contains an element of the form $D(a,b)$ with $a,b \geq 0$ and with at least one of the $a,b$ being nonzero. Since the set $\{D(na,nb)\mid n \geq 1\}$ is not precompact inside $G$, the lemma is proved.
\end{proof}


\begin{lemma}\label{lem:play-with-dim}
Let $q \geq 2$ and write $N = 1+q+q^2$. Let $T$ be a triangle presentation and assume that $\Delta_T$ is the Bruhat--Tits $\tilde{A}_2$-building of a commutative local field $\KK$. Let $G = \AutTr(\Delta_T)$ and $K = \Stab_G(e)$. Denote by $\Gamma_T$ the countable group defined by \eqref{eqn:triangle-group}, generated by $S_+ = \{ a_i \mid i=1,\ldots,N\}$ and $S_- = S_+^{-1}$.
\begin{enumerate}[(i)]
\item If $\cH \in \cC_f(K < G)$ and $\dim(\cH) < N$, then $\cH \in \Rep(K)$.
\item Let $\cH \in \cC_f(K < G)$ with $\dim(\cH) = N$. Assume that $\cH \not\in \Rep(K)$ and that
$$E \in \chi_K \cdot\cH^{\ot 3} = \chi_K \cdot \cH \ot_{L^\infty(G/K)} \cH \ot_{L^\infty(G/K)} \ot \cH$$
is an invariant vector with $\|E\|^2 = |T|$. Then $\supp(\cH)$ equals either $S_+ K$ or $S_- K$. When $\supp(\cH) = S_+ K$, there exist unit vectors $\eta_i \in \chi_K \cdot \cH \cdot \chi_{a_i K}$ and a map $\om : T \recht S^1$ such that
    $$E = \sum_{(i,j,k) \in T} \om(i,j,k) \, (\eta_i \ot \pi_\cH(a_i) \eta_j \ot \pi_\cH(a_i a_j) \eta_k) \; .$$
    When $\supp(\cH) = S_- K$, there exist unit vectors $\eta_i \in \chi_K \cdot \cH \cdot \chi_{a_i^{-1} K}$ and a map $\om : T \recht S^1$ such that
    $$\sum_{(i,j,k) \in T} \om(i,j,k) \, (\eta_k \ot \pi_\cH(a_k^{-1}) \eta_j \ot \pi_\cH(a_k^{-1} a_j^{-1}) \eta_i) \; .$$
\end{enumerate}
\end{lemma}

\begin{proof}
Define the subsets $S_{n,m} \subset \Gamma_T$ as in \eqref{eqn:Snm}. As explained in the proof of part (iii) of Proposition \ref{thm:generic-transitivity}, the double cosets in $K g K$ in $G$ are precisely given as the subsets $S_{n,m} K$. So, the cardinalities of the subsets $K g K/K$ of $G/K$ are $|S_{n,m}|$. We have $S_{0,0} = \{e\}$, $S_{1,0} = S_+$ and $S_{0,1} = S_-$. Note that $|S_{1,0}| = |S_{0,1}| = N$, while for $n,m \geq 1$, we have that $|S_{n,m}| > N$.

(i) Assume that $\cH \in \cC_f(K < G)$ and $\dim(\cH) < N$. Since the categorical dimension of $\cH$ is at least the cardinality of $\supp(\cH)/K$ inside $G/K$, the discussion above forces $\supp(\cH) = K$. This means that $\cH \in \Rep(K)$.

(ii) Take $\cH$ and $E \in \chi_K \cdot \cH^{\ot 3}$ as in the formulation of the lemma. Since $\cH \not\in \Rep(K)$, we have $\supp(\cH) \neq K$. The same dimension argument as in the proof of (i) forces $\supp(\cH)$ to be either $S_+ K$ or $S_- K$. We assume that $\supp(\cH) = S_+ K$ and the proof of the other case is analogous.

View $\chi_K \cdot \cH$ as an irreducible representation of $L^\infty(G/K) \rtimes K$ with support projection $\chi_{S_+ K}$. Since $\chi_K \cdot \cH$ has dimension $N = |S_+|$, the spaces $\cH_1 = \chi_K \cdot \cH \cdot \chi_{a_i K}$ are one-dimensional. Denoting by $P_{1,0}$ the set of positive edges in $\Delta_T$, it follows that $\cH$ is isomorphic with a twist of the $L^\infty(G/K)$-$G$-$L^\infty(G/K)$-module $\ell^2(P_{1,0})$. More precisely, as an $L^\infty(G/K)$-bimodule, we have $\cH \cong \ell^2(P_{1,0})$ and under this isomorphism of bimodules, the unitary representation of $G$ on $\ell^2(P_{1,0})$ is of the form $g \cdot e_p \in S^1 e_{g \cdot p}$ for every $g \in G$ and $p \in P_{1,0}$. We then identify $\cH^{\ot 3}$ with a twist of $\ell^2(P_{3,0})$, where $P_{n,0}$ denotes the set of paths in $\Delta_T$ consisting of $n$ positive edges. By part (iii) of Proposition \ref{thm:generic-transitivity}, the action of $K$ on the positive edges starting at $e$ is $2$-transitive. It follows that the action of $K$ on the set of positive loops of length $3$ starting at $e$ is transitive. Then, the lemma follows.
\end{proof}

Given a triangle presentation $T$ of order $q \geq 2$ on the set $\{1,\ldots,N\}$, where $N = 1 + q + q^2$, we define for every map $\om : T \recht S^1$ the vector $E_{T,\om} \in \CC^N \ot \CC^N \ot \CC^N$ given by
\begin{equation}\label{eq:vector-E-om}
E_{T,\om} = \sum_{(i,j,k) \in T} \om(i,j,k) \, (e_i \ot e_j \ot e_k) \; .
\end{equation}

\begin{lemma}\label{lem:rigid-intertwiners}
Let $q \geq 2$ be an integer and $N = 1 + q + q^2$. Let $T_1$ and $T_2$ be triangle presentations on the set $\{1,\ldots,N\}$ and let $\om : T_2 \recht S^1$ be a function. If $w \in \cU(\CC^N)$ and $(w \ot w \ot w) E_{T_1} = E_{T_2,\om}$, then $w$ is a monomial matrix (i.e.\ in every row and every column of $w$, there is exactly one nonzero element). In particular, the triangle presentations $T_1$ and $T_2$ are isomorphic.
\end{lemma}
\begin{proof}
Write $F = \{1,\ldots,N\}$. As before, we write $i \to_1 j$ if there exists a $k \in F$ with $(i,j,k) \in T_1$. We similarly define the notation $i \to_2 j$. Define the subsets $R_i \subset F \times F$ consisting of all $(a,b)$ with $a \to_i b$. Define $\om : R_2 \recht S^1$ so that $\om(i,j) = \om(i,j,k)$ whenever $(i,j,k) \in T_2$, which is unambiguous since for every $(i,j) \in R_2$, there is a unique $k \in F$ with $(i,j,k) \in T_2$.

Define the operators $Y_i : \CC^N \recht \CC^N \ot \CC^N$ given by
$$Y_1(e_k) = \sum_{i,j \in F , (i,j,k) \in T_1} e_i \ot e_j \qquad\text{and}\qquad Y_2(e_k) = \sum_{i,j \in F , (i,j,k) \in T_2} \om(i,j) \, e_i \ot e_j \; .$$
The assumption $(w \ot w \ot w) E_{T_1} = E_{T_2,\om}$ says that $(w \ot w) Y_1 = Y_2 \overline{w}$. Define $P_i = Y_i Y_i^*$, so that $(w \ot w) P_1 (w^* \ot w^*) = P_2$.

Defining the vector functionals $\Om_{\mu,\rho} : B(\CC^N) \recht \CC : \Om_{\mu,\rho}(T) = \langle T \mu, \rho \rangle$, note that
$$(\id \ot \Om_{e_i,e_i})(P_1) = \sum_{k \in F, k \to_1 i} e_{kk} \; .$$
So for $i=1,\ldots,N$, the operators $(\id \ot \Om_{e_i,e_i})(P_1)$ are commuting orthogonal projections of rank $1+q$ with the property that the product of two distinct of these projections has rank $1$. Writing
$$L_i := w \, (\id \ot \Om_{e_i,e_i})(P_1) \, w^* \; ,$$
also $L_1,\ldots,L_N$ are commuting orthogonal projections with the property that $L_i L_j$ is a rank one projection for all $i \neq j$.

Write $\eta_i = w(e_i)$, so that $\eta_{i,j} = w_{ji}$. Note that $L_i = (\id \ot \Om_{\eta_i,\eta_i})(P_2)$. It follows that
$$L_i = \sum_{\substack{a,b,c,d,k \in F \\ (a,b,k) \in T_2 \;\text{and}\;(c,d,k) \in T_2}} \; \om(a,b) \; \overline{\om(c,d)} \; \overline{w_{bi}} \; w_{di} \; \ e_{ac} \; .$$
So, the diagonal elements of $L_i$ are given by
\begin{equation}\label{eq:diagonal-Li}
L_{i,aa} = \sum_{b \in F , a \to_2 b} \; |w_{bi}|^2 \; .
\end{equation}
Define $F_2 \subset F \times F$ consisting of all pairs $(a,b)$ with $a \neq b$. Since $T_2$ is a triangle presentation, for every $(b,d) \in F_2$, there is a unique pair $(\al(b,d),\gamma(b,d)) \in F_2$ such that there exists a (also unique) $k \in F$ satisfying $(\al(b,d),b,k) \in T_2$ and $(\gamma(b,d),d,k) \in T_2$. The off-diagonal elements of $L_i$ are then given by
\begin{equation}\label{eq:off-diagonal-Li}
L_{i,\al(b,d)\gamma(b,d)} = \om(\al(b,d),b) \; \overline{\om(\gamma(b,d),d)} \; \overline{w_{bi}} \; w_{di} \; .
\end{equation}
Fix $i \neq j$. Since $L_i$ and $L_j$ are commuting orthogonal projections and $L_i L_j$ has rank $1$, we get that $1 = \Tr(L_i L_j) = \Tr(L_i L_j^*)$. Since $(b,d) \mapsto (\al(b,d),\gamma(b,d))$ is a permutation of $F_2$, we get that
$$1 = \Tr(L_i L_j^*) = \sum_{a \in F} \; L_{i,aa} \; \overline{L_{j,aa}} \;\; + \;\; \sum_{(b,d) \in F_2} \; L_{i,\al(b,d)\gamma(b,d)} \; \overline{L_{j,\al(b,d)\gamma(b,d)}} \; .$$
Using \eqref{eq:diagonal-Li} and \eqref{eq:off-diagonal-Li}, we conclude that
\begin{equation}\label{eq:sum-sum}
1 = \sum_{\substack{a,b,d \in F \\ a \to_2 b , a \to_2 d}} \; |w_{bi}|^2 \; |w_{dj}|^2  \;\; + \;\; \sum_{(b,d) \in F_2} \; \overline{w_{bi}} \; w_{di} \; w_{bj} \; \overline{w_{dj}} \; .
\end{equation}
When $b = d$, there are precisely $1+q$ elements $a \in F$ with $a \to_2 b$ and $a \to_2 d$. When $b \neq d$, there is precisely one such element $a \in F$. So the first term on the right hand side of \eqref{eq:sum-sum} can be rewritten as
$$(1+q) \, \sum_{b \in F} \; |w_{bi}|^2 \; |w_{bj}|^2 \;\; + \;\; \sum_{(b,d) \in F_2} \; |w_{bi}|^2 \; |w_{dj}|^2 \; .$$
Writing the sum over $(b,d) \in F_2$ as the sum over all $(b,d) \in F \times F$ minus the sum over $(b,b)$, $b \in F$ and using that $w$ is a unitary, the above expression becomes
$$q \, \sum_{b \in F} \; |w_{bi}|^2 \; |w_{bj}|^2 \;\; + \;\; 1 \; .$$
Also in the second term on the right hand side of \eqref{eq:sum-sum}, we rewrite the sum over $(b,d) \in F_2$ as the sum over all $(b,d) \in F \times F$ minus the sum over $(b,b)$, $b \in F$. Using that $w$ is unitary and $i \neq j$, this second term is equal to
$$- \sum_{b \in F} \; |w_{bi}|^2 \; |w_{bj}|^2 \; .$$
Altogether, we have proved that for all $i \neq j$,
$$1 = 1 + (q-1) \, \sum_{b \in F} \; |w_{bi}|^2 \; |w_{bj}|^2 \; .$$
Since $q \geq 2$, it follows that for every $b \in F$, there is at most one $i \in F$ with $w_{bi} \neq 0$. Since $w$ is unitary, it follows that $w$ is a monomial matrix. Defining the permutation $\theta : F \recht F$ so that $w_{\theta(i)i} \neq 0$ for all $i \in F$, the equality $(w \ot w \ot w) E_{T_1} = E_{T_2,\om}$ implies that $(\theta \times \theta \times \theta)(T_1) = T_2$, so that the triangle presentations $T_1$ and $T_2$ are isomorphic.
\end{proof}

\section{\boldmath Diagram calculus for \texorpdfstring{$\Rep(\GG_T)$}{Rep(G\_T)}}\label{sec:diagram-calculus}

In Definition \ref{def:G-T}, we associated to every triangle presentation $T$ the compact quantum group $\GG_T$ with fundamental representation $u$. In this section, we provide a graphical description of the intertwiner spaces $(u^{\ot n},u^{\ot m})$ in order to approach the following open problem.

\begin{question}\label{rem:monoidal-equivalence}
Let $T$ and $T'$ be triangle presentations of the same order $q \geq 2$. Are $\GG_T$ and $\GG_{T'}$ monoidally equivalent?
\end{question}

Below, we provide convincing evidence that at least for low values of $q$ and, in particular, for the two triangle presentations $T_1$ and $T_2$ in Table \ref{tbl:triang-pres-examples}, the previous question has a positive answer; see Remark \ref{rem:indication-mon-equiv}.

\subsection{Intertwiner diagrams}

Analogously to the diagram calculus of \cite{woron-tkdual}, we give a graphical description of a set of intertwiners in $(u^{\ot n},u^{\ot m})$ and we explain how to compute the matrix coefficients of these intertwiners.

\begin{definition}
Fix $n\in\NN$. An $\NC_3^\circ$\textit{-partition of size} $n$ is a partition of a totally ordered set $S$ of $n$ points such that

\begin{enumerate}[(i)]
\item $p$ is noncrossing, i.e.\@ if $i<j<k<l$ are points in $S$ with $i$ and $k$ in the same block, then $j$ and $l$ lie in different blocks;
\item all blocks of $p$ are of size $3$ or singletons;
\item singleton blocks are not nested, i.e.\@ if $i<j<k$ with $i$ and $k$ in the same block, then the singleton $\{j\}$ is not a block of $p$.
\end{enumerate}

Consider a triangle presentation $T$ over $F$. Given an $\NC_3^\circ$-partition $p$ on $S$, a $T$-\textit{labeling} of $p$ is a function $\lambda:S\to F$ such that $(\lambda(i),\lambda(j),\lambda(k))\in T$ whenever $i<j<k$ and $\{i,j,k\}$ is a block of size $3$ in $p$.
\end{definition}

We usually represent the points of $S$ by equidistant points on a horizontal line, and draw lines between points to indicate which points belong to the same block (see Figure~\ref{fig:nc3circ-example}).

\begin{figure}[hb]
        \[
            \begin{tikzpicture}[scale=0.6,yscale=0.5, thick, every circle node/.style={draw, fill=black!10, minimum width=2pt, inner sep=2pt}]
                \node[circle] (pt1) at (1,0) [draw] {};
                \node[circle] (pt2) at (2,0) [draw] {};
                \node[circle] (pt3) at (3,0) [draw] {};
                \draw (pt1) -- (1,2) -- (3,2) -- (pt3);
                \draw (pt2) -- (2,2);
            \end{tikzpicture}
            \qquad\qquad\qquad\qquad
            \begin{tikzpicture}[scale=0.6,yscale=0.5, thick, every circle node/.style={draw, fill=black!10, minimum width=2pt, inner sep=2pt}]
                \node[circle] (pt1) at (1,0) [draw] {};
                \node[circle] (pt2) at (2,0) [draw] {};
                \node[circle] (pt3) at (3,0) [draw] {};
                \node[circle] (pt4) at (4,0) [draw] {};
                \node[circle] (pt5) at (5,0) [draw] {};
                \node[circle] (pt6) at (6,0) [draw] {};
                \node[circle] (pt7) at (7,0) [draw] {};

                \draw (pt2) -- (2,3) -- (7,3) -- (pt7);
                \draw (pt6) -- (6,3);
                \draw (pt3) -- (3,2) -- (5,2) -- (pt5);
                \draw (pt4) -- (4,2);
            \end{tikzpicture}
        \]
        \caption{Two examples of $\NC_3^\circ$-partitions}\label{fig:nc3circ-example}
\end{figure}

\begin{definition}\label{def:diagram-calculus}
        Fix $n,m\in\NN$.
        An \textit{intertwiner diagram of type} $(n,m)$ is a tuple $(p_\ell,p_u,S)$, where
        \begin{enumerate}[(i)]
            \item $p_u,p_\ell$ are $\NC_3^\circ$-partitions of the same finite totally ordered set $S$,
            \item $p_u$ has $m$ singleton blocks,
            \item $p_\ell$ has $n$ singleton blocks.
        \end{enumerate}
        We refer to $p_u$ as the \textit{upper} and $p_\ell$ as the \textit{lower} partition (see Figure~\ref{fig:intertwiner-diagram-example}).
        Consider a triangle presentation $T$ with base set $F$.
        Fix an intertwiner diagram $D=(p_\ell,p_u,S)$ of type $(n,m)$, and choose multi-indices $I\in F^n$, $J\in F^m$.
        A $T$-\textit{labeling} of $D$ compatible with $(I,J)$ is a function $\lambda:S\to F$ such that
        \begin{enumerate}[(i)]
            \item $\lambda$ is a $T$-labeling of both $p_\ell$ and $p_u$;
            \item if $i_1,\ldots,i_m$ are the singleton blocks of $p_\ell$ in increasing order, $I=(\lambda(i_1),\ldots,\lambda(i_m))$;
            \item if $j_1,\ldots,j_n$ are the singleton blocks of $p_u$ in increasing order, $J=(\lambda(j_1),\ldots,\lambda(j_n))$.
        \end{enumerate}
        This procedure defines a linear map $[D]_T:(\CC^F)^{\otimes m}\to (\CC^F)^{\otimes n}$ by setting
        \begin{equation}
        \label{eqn:labelling-formula}
            \langle [D]_T e_J, e_I\rangle = \#\{T\text{-labelings of } D \text{ compatible with } (I,J)\}\;.
        \end{equation}
\end{definition}

\begin{figure}[ht]
            \[
                \begin{tikzpicture}[scale=0.8,yscale=0.5,thick, every circle node/.style={draw, fill=black!10, minimum width=2pt, inner sep=2pt}]
                    \node[circle] (out1) at (4,0) [draw] {};
                    \node[circle] (out2) at (8,0) [draw] {};
                    \node[circle] (out3) at (9,0) [draw] {};
                    \node[circle] (in1) at (1,0) [draw] {};
                    \node[circle] (in2) at (2,0) [draw] {};
                    \node[circle] (in3) at (6,0) [draw] {};

                    \draw (in1) -- (1,-2) -- (3,-2) -- (3,2) -- (5,2) -- (5,-2) -- (7,-2) -- (7,2) -- (9,2) -- (out3);
                    \draw (in2) -- (2,-2);
                    \draw (out1) -- (4,2);
                    \draw (in3) -- (6,-2);
                    \draw (out2) -- (8,2);
                    \node[circle] (inn1) at (3,0) [draw] {};
                    \node[circle] (inn2) at (5,0) [draw] {};
                    \node[circle] (inn3) at (7,0) [draw] {};
                \end{tikzpicture}
            \]
        \caption{Example of an intertwiner diagram, representing the intertwiner $(E^*\otimes 1\otimes E^*\otimes 1 \ot 1)(1 \ot 1 \otimes E\otimes 1\otimes E)$. If we label the points from left to right by $\{1,\ldots,9\}$, the upper partition is $\{\{1\},\{2\},\{3,4,5\},\{6\},\{7,8,9\}\}$, and the lower partition is $\{\{1,2,3\},\{4\},\{5,6,7\},\{8,9\}\}$.}
        \label{fig:intertwiner-diagram-example}
\end{figure}
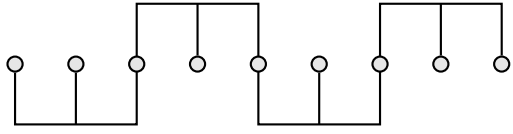

The next lemma says that the diagram calculus provides a universal description of the intertwiner spaces $(u^{\ot n},u^{\ot m})$, independently of the choice of triangle presentation $T$. The precise proof is straightforward but tedious. So we only provide a sketch.

\begin{lemma}\label{thm:diagram-calculus-props}
The diagram calculus of Definition~\ref{def:diagram-calculus} has the following properties:
\begin{enumerate}[(i)]
        \item For every intertwiner diagram $D$, there exists an intertwiner diagram $D'$ such that $[D]_T^*=[D']_T$ for any triangle presentation $T$.
        \item For every two intertwiner diagrams $A$ and $B$ there exists an intertwiner diagram $C$ such that $[A]_T\otimes [B]_T=[C]_T$ for any triangle presentation $T$.
        \item For every intertwiner diagram $A$ of type $(l,m)$ and every intertwiner diagram $B$ of type $(n,l)$, there exists an intertwiner diagram $C$ of type $(n,m)$ such that $[B]_T\circ [A]_T = [C]_T$ for any triangle presentation $T$.
        \item For any triangle presentation $T$, the linear maps induced by intertwiner diagrams of type $(n,m)$ span the intertwiner space $(u^{\otimes n},u^{\otimes m})$.
\end{enumerate}
\end{lemma}
\begin{proof}[Sketch of proof]
To prove (i), it suffices to exchange the upper and lower partitions of $D$. Statement (ii) follows by juxtaposing the respective upper and lower partitions of $A$ and $B$.

The diagram $C$ in statement~(iii) can be visualized as follows. it suffices to identify the singletons of the lower partition of $A$ with the singletons of the upper partition of $B$ and rearrange the points to get a new intertwiner diagram $C$ as in Figure~\ref{fig:composition-example}. Label the newly identified points as $c_1,\ldots,c_l$.
Then, given any triangle presentation $T$ with base set $F$, one checks that
\begin{align*}
\langle &[C]_T e_J, e_I\rangle\\ &= \sum_{K\in F^l} \#\{T\text{-labelings } \lambda \text{ of } C \text{ compatible with } (I,J) \text{ such that } K=(\lambda(c_1),\ldots,\lambda(c_l))\}\\
&=\sum_{K\in F^l} \langle [B]_T e_K, e_I\rangle \langle [A]_T e_J, e_K\rangle\\
&=\langle [B]_T[A]_T e_J, e_I\rangle\;,
\end{align*}
where $I,J$ are arbitrary multi-indices in $F^n$ and $F^m$, respectively.

To prove (iv), note that if $D$ is the left diagram of Figure \ref{fig:nc3circ-example}, then $[D]_T : \CC \recht (\CC^N)^{\ot 3}$ is given by the invariant vector $E$. By induction on the number of blocks of size three, it follows that $[D]_T$ is an intertwiner for every intertwiner diagram $D$. Since $E\in (uuu,\varepsilon)$ and $1\in (u,u)$ can be realized by diagrams, it follows from (i), (ii), (iii) and Remark \ref{rem:Tannaka-point-of-view} that the intertwiner spaces are linearly spanned by the operators $[D]_T$.
\end{proof}

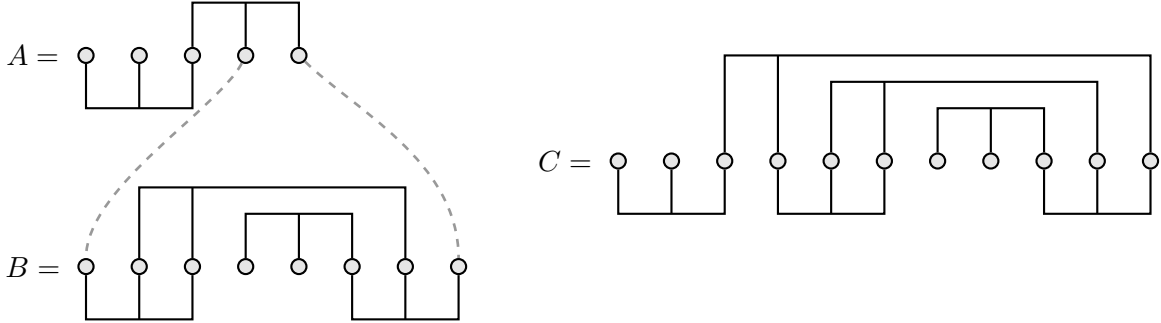
\begin{figure}[ht]
    \[
            \begin{tikzpicture}[scale=0.7,yscale=0.5, thick, every circle node/.style={draw, fill=black!10, minimum width=2pt, inner sep=2pt}]
                    \node at (0,8) {$A=$};
                    \node[circle] (pta1) at (1,8) [draw] {};
                    \node[circle] (pta2) at (2,8) [draw] {};
                    \node[circle] (pta3) at (3,8) [draw] {};
                    \node[circle] (pta4) at (4,8) [draw] {};
                    \node[circle] (pta5) at (5,8) [draw] {};
                    \draw (pta3) -- (3,10) -- (5,10) -- (pta5);
                    \draw (pta4) -- (4,10);
                    \draw (pta1) -- (1,6) -- (3,6) -- (pta3);
                    \draw (pta2) -- (2,6);

                    \node at (0,0) {$B=$};
                    \node[circle] (ptb1) at (1,0) [draw] {};
                    \node[circle] (ptb2) at (2,0) [draw] {};
                    \node[circle] (ptb3) at (3,0) [draw] {};
                    \node[circle] (ptb4) at (4,0) [draw] {};
                    \node[circle] (ptb5) at (5,0) [draw] {};
                    \node[circle] (ptb6) at (6,0) [draw] {};
                    \node[circle] (ptb7) at (7,0) [draw] {};
                    \node[circle] (ptb8) at (8,0) [draw] {};
                    \draw (ptb1) -- (1,-2) -- (3,-2) -- (ptb3);
                    \draw (ptb2) -- (2,-2);
                    \draw (ptb2) -- (2,3) -- (7,3) -- (ptb7);
                    \draw (ptb3) -- (3,3);
                    \draw (ptb4) -- (4,2) -- (6,2) -- (ptb6);
                    \draw (ptb5) -- (5,2);
                    \draw (ptb6) -- (6,-2) -- (8,-2) -- (ptb8);
                    \draw (ptb7) -- (7,-2);

                    \draw [line width=1pt, dashed, color=black!40] (pta4) .. controls (3.5,6) and (1,3) .. (ptb1);
                    \draw [line width=1pt, dashed, color=black!40] (pta5) .. controls (6,6) and (8,4) .. (ptb8);

                    \begin{scope}[xshift=10cm,yshift=12cm]
                    \node at (0,-8) {$C=$};
                    \node[circle] (ptc1) at (1,-8) [draw] {};
                    \node[circle] (ptc2) at (2,-8) [draw] {};
                    \node[circle] (ptc3) at (3,-8) [draw] {};
                    \node[circle] (ptc4) at (4,-8) [draw] {};
                    \node[circle] (ptc5) at (5,-8) [draw] {};
                    \node[circle] (ptc6) at (6,-8) [draw] {};
                    \node[circle] (ptc7) at (7,-8) [draw] {};
                    \node[circle] (ptc8) at (8,-8) [draw] {};
                    \node[circle] (ptc9) at (9,-8) [draw] {};
                    \node[circle] (ptc10) at (10,-8) [draw] {};
                    \node[circle] (ptc11) at (11,-8) [draw] {};

                    \draw (ptc1) -- (1,-10) -- (3,-10) -- (ptc3);
                    \draw (ptc2) -- (2,-10);
                    \draw (ptc4) -- (4,-10) -- (6,-10) -- (ptc6);
                    \draw (ptc5) -- (5,-10);
                    \draw (ptc3) -- (3,-4) -- (11,-4) -- (ptc11);
                    \draw (ptc4) -- (4,-4);
                    \draw (ptc5) -- (5,-5) -- (10,-5) -- (ptc10);
                    \draw (ptc6) -- (6,-5);
                    \draw (ptc7) -- (7,-6) -- (9,-6) -- (ptc9);
                    \draw (ptc8) -- (8,-6);
                    \draw (ptc9) -- (9,-10) -- (11,-10) -- (ptc11);
                    \draw (ptc10) -- (10,-10);
                    \end{scope}
                \end{tikzpicture}
        \]
        \caption{Composition of two intertwiner diagrams $A$ and $B$ resulting in a third intertwiner diagram $C$.}
        \label{fig:composition-example}
\end{figure}


\subsection{\boldmath Towards monoidal equivalence of \texorpdfstring{$\GG_T$}{G\_T}}

The simplest intertwiner diagrams are those given by a noncrossing partition with block size $3$. Denote by $\NC_3(n)$ the set of all partitions of $3n$ points in $n$ noncrossing blocks of size $3$. For every triangle presentation $T$ with associated compact quantum group $\GG_T$ and for every $p \in \NC_3(n)$, the vector $[p]_T$ is an invariant vector of $u^{\ot 3n}$. In particular, the unique element in $\NC_3(1)$, i.e.\ the partition on the left of Figure~\ref{fig:nc3circ-example}, corresponds to the invariant vector $E\in u^{\otimes 3}$. Note however that for all $n \geq 2$, the representation $u^{\otimes 3n}$ contains more invariant vectors than those spanned by partitions in $\NC_3(n)$.

Even though the partitions in $p \in \NC_3(n)$ are not enough to produce all invariant vectors of $u^{\ot 3n}$, the following proposition says that the scalar products between the vectors $[p]_T$ suffice to determine the entire monoidal structure of $\Rep(\GG_T)$. This reduces Question~\ref{rem:monoidal-equivalence} to a very concrete combinatorial problem, which we could however not solve.

\begin{proposition}\label{thm:inn-prod-monoidal-equivalence}
Let $T$ and $T'$ be triangle presentations of the same order $q\geq 2$. Assume that
\begin{equation}\label{eqn:inn-prods-equal}
\langle [p_1]_T, [p_2]_T\rangle = \langle [p_1]_{T'}, [p_2]_{T'}\rangle\;,
\end{equation}
for all $n\in\NN$ and all $p_1,p_2\in\NC_3(n)$. Then $\GG_T$ and $\GG_{T'}$ are monoidally equivalent.
\end{proposition}

\begin{proof}
We denote the fundamental representation of $\GG_T$ by $u$ and the one of $\GG_{T'}$ by $u'$.
The assumption \eqref{eqn:inn-prods-equal} can be rephrased as follows: for any intertwiner diagram $D_0$ of type $(0,0)$, $[D_0]_T=[D_0]_{T'}$.

We first show that $\Tr_{u^{\otimes n}}([D]_T)=\Tr_{u'^{\otimes n}}([D]_{T'})$ for all diagrams $D$ of type $(n,n)$.
To this end, inductively define the sequence $p_n \in \NC_3(n)$ by defining $p_0$ as the empty partition and
\[
        p_n=\quad\begin{tikzpicture}[scale=0.8,baseline={([yshift=-1.2ex]current bounding box.center)},thick, yscale=0.5, every circle node/.style={draw, fill=black!10, minimum width=2pt, inner sep=2pt}]
				\node[circle] (outer1) at (0,0) [draw] {};
                \roundedbox{(2,1)}{1}{0}{$p_{n-1}$};
				\node[circle] (midleg) at (4,0) [draw] {};
				\node[circle] (outer2) at (5,0) [draw] {};
				\coordinate (outer1b) at (0,3);
                \coordinate (midlegb) at (4,3);
				\coordinate (outer2b) at (5,3);
				\draw (outer1) -- (outer1b) -- (outer2b) -- (outer2);
                \draw (midlegb) -- (midleg);
			\end{tikzpicture}
\]
For any intertwiner $V\in (u^{\otimes n}, u^{\otimes n})$, we have that
\[
(q+1)^n\Tr_{u^{\otimes n}}(V)=[p_n]_T^* (V\otimes 1^{\otimes 2n}) [p_n]_T \; .
\]
By Lemma~\ref{thm:diagram-calculus-props}, we can associate to every intertwiner diagram $D$ of type $(n,n)$, an intertwiner diagram $D_{\Tr}$ of type $(0,0)$ such that
$$[p_n]_T^* ([D]_T \otimes 1^{\otimes 2n}) [p_n]_T = [D_{\Tr}]_T \quad\text{and}\quad [p_n]_{T'}^* ([D]_{T'}\otimes 1^{\otimes 2n}) [p_n]_{T'} = [D_{\Tr}]_{T'} \; .$$
Since by assumption, $[D_{\Tr}]_T = [D_{\Tr}]_{T'}$, it follows that
\begin{equation}\label{eqn:trace-preserved}
\Tr_{u^{\otimes n}}([D]_T)= \Tr_{u'^{\otimes n}}([D]_{T'}) \; .
\end{equation}

If $A$ and $B$ are intertwiner diagrams of type $(m,n)$, then Lemma~\ref{thm:diagram-calculus-props} gives us an intertwiner diagram $C$ of type $(n,n)$ such that $[B]_T^* [A]_T=[C]_T$ and $[B]_{T'}^* [A]_{T'}=[C]_{T'}$. Hence \eqref{eqn:trace-preserved} implies that $\langle [A]_T,[B]_T\rangle=\langle [A]_{T'},[B]_{T'}\rangle$.
This means that there are unique and well defined unitary linear maps
$$\theta_{m,n} : (u^{\otimes m},u^{\otimes n}) \recht (u'^{\otimes m},u'^{\otimes n})$$
satisfying $\theta_{m,n}([D]_T) = [D]_{T'}$ for all intertwiner diagrams $D$ of type $(m,n)$.

Using the diagrammatic description of composition, involution and tensor products of intertwiners in Lemma~\ref{thm:diagram-calculus-props}, it follows that the unitary linear maps $\theta_{m,n}$ respect composition, involution and tensor products of intertwiners. So these maps give rise to a unitary monoidal equivalence from the full subcategory of $\Rep(\GG_T)$ spanned by the objects $u^{\otimes n}$, $n\in\NN$, to the full subcategory of $\Rep(\GG_{T'})$ spanned by the objects ${u'}^{\otimes n}$, $n\in\NN$. By universality (see \cite[pp.~71--72]{neshveyev-tuset}), this extends to a unitary monoidal equivalence between $\Rep(\GG_T)$ and $\Rep(\GG_{T'})$.
\end{proof}

\begin{remark}\label{rem:indication-mon-equiv}
By Proposition \ref{thm:inn-prod-monoidal-equivalence}, in order to answer Question \ref{rem:monoidal-equivalence} positively and hence prove that $\Rep(\GG_T)$ only depends on the order of $T$, it suffices to prove that the scalar products $\langle [p_1]_T, [p_2]_T \rangle$ only depend on the partitions $p_1,p_2 \in \NC_3$ and the order of $T$.

While we could not prove this in general, checking \eqref{eqn:inn-prods-equal} for concrete partitions is a matter of counting solutions to a constraint satisfaction problem, which is easy to implement on a computer. We have tested \eqref{eqn:inn-prods-equal} for all pairs $p_1,p_2\in\NC_3(n)$ with $n\leq 9$ and for the triangle presentations of order $2$ numbered $A.1$, $A.1'$ and $B.1$ in \cite{cmsz2}. For $q=3$, we numerically verified \eqref{eqn:inn-prods-equal} for all pairs in $p_1,p_2\in\NC_3(n)$ with $n\leq 7$ and a selection of triangle presentations, including some triangle presentations with exotic buildings (nos.\@ $1.1$, $1.1'$, $22.1$ and $36.1$ in \cite{cmsz2}).

This provides some empirical justification to conjecture that, at least for small values of $q$, the answer to Question \ref{rem:monoidal-equivalence} is yes.

Due to computational limitations, we have not tried any higher-order triangle presentations. It appears that triangle presentations can behave very exotically for higher values of $q$ (see e.g.\@ \cite{radu-nondesarg-pres}) and we refrain from making any claims about $q>3$ here.

Another approach to answer Question \ref{rem:monoidal-equivalence} could be to use Remark \ref{rem:planar-algebra} as a starting point to give a graphical description of the planar algebra defined by a triangle presentation. One might then hope to use planar algebra methods to prove that these planar algebras only depend, up to isomorphism, on the order of the triangle presentation.
\end{remark}



\begin{thebibliography}{CMSZ93b}\setlength{\itemsep}{-1mm} \setlength{\parsep}{0mm} \small


\bibitem[Ara14]{arano-suqn}
Y.~Arano, \emph{Unitary spherical representations of Drinfeld doubles}, J. Reine Angew. Math. (2016), 1435--5345.

\bibitem[Ara16]{arano-drinfeld-unitary-duals}
Y.~Arano, \emph{Comparison of unitary duals of Drinfeld doubles and complex semisimple Lie groups}, Commun. Math. Phys. \textbf{351} (2017), 1137--1147.

\bibitem[AdLW17]{arano-delaat-wahl}
Y.~Arano, T.~de Laat and J.~Wahl, \emph{The Fourier algebra of a rigid $C^*$-tensor category}, Publ. Res. Inst. Math. Sci. \textbf{54} (2018), 393--410.

\bibitem[AV16]{arano-vaes}
Y.~Arano and S.~Vaes, \emph{C*-tensor categories and subfactors for totally disconnected groups}, in ``Operator Algebras and Applications'', Abel Symposium 2015, Springer, 2016, pp.\ 1--43.

\bibitem[BS92]{baaj-skandalis}
S.~Baaj and G.~Skandalis, \emph{Unitaires multiplicatifs et dualit\'{e} pour les produits crois\'{e}s de $C^*$-alg\`{e}bres}, Ann. Sci. \'{E}cole Norm. Sup. (4) {\bf 26} (1993), 425--488.


\bibitem[BP98]{bisch-popa}
D.~Bisch and S.~Popa, \emph{Examples of subfactors with property T standard invariant}, Geom. Funct. Anal. {\bf 9} (1999), 215--225.

\bibitem[Bou16]{bouljihad}
M.~Bouljihad, \emph{Existence of rigid actions for finitely-generated non-amenable linear groups}, preprint. \href{http://arxiv.org/abs/1606.05911}{arXiv:1606.05911}

\bibitem[BO08]{brown-ozawa}
N.P. Brown and N.~Ozawa, \emph{C*-algebras and finite-dimensional approximations}, Graduate Studies in Mathematics {\bf 88}, American Mathematical Society, 2008.


\bibitem[CMSZ91a]{cmsz1}
D.I. Cartwright, A.M. Mantero, T.~Steger, and A.~Zappa, \emph{Groups acting simply transitively on the vertices of a building of type $\tilde{A}_2$,
  I}, Geometriae Dedicata \textbf{47} (1993), 143--166.

\bibitem[CMSZ91b]{cmsz2}
D.I. Cartwright, A.M. Mantero, T.~Steger, and A.~Zappa, \emph{Groups acting simply transitively on the vertices of a building of type $\tilde{A}_2$, II}, Geometriae Dedicata \textbf{47} (1993), 167--223.

\bibitem[CMS93]{cms}
D.I. Cartwright, W.~M{\l}otkowski, and T.~Steger, \emph{Property {(T)} and $\tilde{A}_2$ groups}, Ann. Inst. Fourier \textbf{44} (1994), 213--248.

\bibitem[D86]{drinfeld}
V.G.~Drinfel'd, \emph{Quantum groups}, in ``Proceedings of the International Congress of Mathematicians (Berkeley, 1986)'', Amer. Math. Soc., Providence, 1987, pp.\ 798--820.

\bibitem[FH64]{feit-higman}
W.~Feit and G.~Higman, \emph{The nonexistence of certain generalized polygons}, J. Algebra \textbf{1} (1964), 114--131.

\bibitem[Fim08]{fima-prop-T}
P.~Fima, \emph{Kazhdan's property {(T)} for discrete quantum groups}, Internat. J. Math. \textbf{21} (2010), 47--65.

\bibitem[FMP15]{fima-bicrossed}
P.~Fima, K.~Mukherjee and I.~Patri, \emph{On compact bicrossed products}, J. Noncommut. Geom. \textbf{11} (2017), 1521--1591.

\bibitem[GJ15]{ghosh-jones}
S.K. Ghosh and C.~Jones, \emph{Annular representation theory for rigid $C^*$-tensor categories}, J. Funct. Anal. \textbf{270} (2016), 1537--1584.

\bibitem[Gur90]{gurevich}
D.I. Gurevich, \emph{Algebraic aspects of the quantum Yang-Baxter equation}, Leningrad Math. J. {\bf 2} (1991), 801--828.

\bibitem[Ji85]{jimbo}
M.~Jimbo, \emph{A $q$-difference analogue of $U(\mathfrak{g})$ and the Yang-Baxter equation}, Lett. Math. Phys. {\bf 10} (1985), 63--69.

\bibitem[Jon99]{jones:planar}
V.F.R.~Jones, \emph{Planar algebras}, preprint, Berkeley 1999, \href{http://arxiv.org/abs/math.QA/9909027}{arXiv:math.QA/9909027}

\bibitem[Jon15]{corey-jones}
C. Jones, \emph{Quantum $G_2$ categories have property (T)}, Internat. J. Math. {\bf 27} (2016), art.\ id.\ 1650015, 23 pp.


\bibitem[Kli60]{klingenberg}
W.~Klingenberg, \emph{Lineare Gruppen \"{u}ber lokalen Ringen}, Amer. J. Math. {\bf 83} (1961), 137--153.


\bibitem[MPW15]{descent-in-buildings}
B.~M\"{u}hlherr, H.P.~Petersson and R.M.~Weiss, \emph{Descent in buildings}, Annals of Mathematics Studies {\bf 190}, Princeton University Press, Princeton, NJ, 2015.

\bibitem[Kac68]{Kac}
G.I.~Kac, \emph{Extensions of groups to ring groups}, Math. USSR Sbornik {\bf 5} (1968), 451--474.

\bibitem[NT10]{nesh-tus-cohom}
S.~Neshveyev and L.~Tuset, \emph{On second cohomology of duals of compact groups}, Internat. J. Math. {\bf 22} (2011), 1231--1260.

\bibitem[NT13]{neshveyev-tuset}
S.~Neshveyev and L.~Tuset, \emph{Compact quantum groups and their representation categories}, Cours Sp\'{e}cialis\'{e}s, no.~20, Soci\'{e}t\'{e} Math\'{e}matique de France, 2013.

\bibitem[NY15]{neshveyev-yamashita}
S.~Neshveyev and M.~Yamashita, \emph{Drinfeld center and representation theory for monoidal categories}, Commun. Math. Phys. \textbf{345} (2016), 385--434.

\bibitem[Ocn93]{ocneanu-chirality}
A.~Ocneanu, \emph{Chirality for operator algebras}, in ``Subfactors (Kyuzeso, 1993)'', World Sci. Publ., River Edge, 1994, pp.\ 39--63.

\bibitem[Pin97]{pink}
R.~Pink, \emph{Compact subgroups of linear algebraic groups}, J. Algebra {\bf 206} (1998), 438--504.


\bibitem[Pop92]{popa-classification-amenable}
S.~Popa, \emph{Classification of amenable subfactors of type II}, Acta Math. {\bf 172} (1994), 163--255.

\bibitem[Pop94]{popa-lambda-lattices}
S.~Popa, \emph{An axiomatization of the lattice of higher relative commutants of a subfactor}, Invent. Math. {\bf 120} (1995), 427--445.

\bibitem[Pop97]{popa-symmenv}
S.~Popa, \emph{Some properties of the symmetric enveloping algebra of a subfactor, with applications to amenability and property~(T)}, Doc. Math.
  \textbf{4} (1999), 665--744.


\bibitem[PSV15]{psv-cohom}
S.~Popa, D.~Shlyakhtenko and S.~Vaes, \emph{Cohomology and $L^2$-Betti numbers for subfactors and quasi-regular
  inclusions}, {\bf 2018}, no.\ 8 (2018), 2241-2331.

\bibitem[PV14]{pv-repr-subfactors}
S.~Popa and S.~Vaes, \emph{Representation theory for subfactors, $\lambda$-lattices and $C^*$-tensor categories}, Commun. Math. Phys.
  \textbf{340} (2015), 1239--1280.

\bibitem[Rad16]{radu-nondesarg-pres}
N.~Radu, \emph{A lattice in a residually non-Desarguesian $\tilde{A}_2$-building},  Bull. Lond. Math. Soc. \textbf{49} (2017), 274--290.

\bibitem[Ron89]{ronan}
M.~Ronan, \emph{Lectures on buildings}, Perspectives in Mathematics, no.~7, Academic Press, 1989.


\bibitem[VV01]{vaes-vainerman}
S.~Vaes and L.~Vainerman, \emph{Extensions of locally compact quantum groups and the bicrossed product construction}, Adv. Math. {\bf 175} (2003), 1--101.

\bibitem[VD94]{van-daele}
A. Van Daele, \emph{Discrete quantum groups}, J. Algebra {\bf 180} (1996), 431–-444.

\bibitem[Wor86]{woron-su2}
S.L. Woronowicz, \emph{Twisted $\SU(2)$ group. An example of a noncommutative differential calculus}, Publ. Res. Inst. Math. Sci. {\bf 23} (1987), 117--181.

\bibitem[Wor88]{woron-tkdual}
S.L. Woronowicz, \emph{Tannaka--Krein duality for compact matrix pseudogroups. Twisted $\mathrm{SU}(N)$-groups}, Invent. Math. \textbf{93} (1988), 35--76.

\bibitem[Wor98]{woron-cqg}
S.L. Woronowicz, \emph{Compact quantum groups}, in ``Sym\'{e}tries quantiques (Les Houches, 1995)'', North-Holland, Amsterdam, 1998, pp.\ 845--884.

\bibitem[{\.Z}uk01]{zuk-propT}
A.~{\.Z}uk, \emph{Property~(T) and Kazhdan constants for discrete groups}, Geom. Funct. Anal. \textbf{13} (2003), 643--670.
\end{thebibliography}
\end{document}